\documentclass[11pt]{amsart}
\usepackage{amssymb,mathrsfs}
\usepackage[all]{xy}

\usepackage[bookmarks=false]{hyperref}

\newcommand{\map}[1]{\xrightarrow{#1}}

\newcommand{\iso}{\cong}
\newcommand{\define}{\stackrel{\mathrm{def}}{=}}

\DeclareMathOperator{\Hom}{Hom}
\DeclareMathOperator{\Aut}{Aut}

\DeclareMathOperator{\Spec}{Spec}

\DeclareMathOperator{\GL}{GL}

\DeclareMathOperator{\Herm}{Herm}

\DeclareMathOperator{\Sht}{Sht}
\DeclareMathOperator{\Ch}{Ch}
\DeclareMathOperator{\Prym}{Prym}
\DeclareMathOperator{\rank}{rank}

\DeclareMathOperator{\Pic}{Pic}

\newcommand{\Q}{\mathbb Q}

\newcommand{\C}{\mathbb C}

\newcommand{\A}{\mathbb A}
\newcommand{\bL}{\mathbb L}

\newcommand{\co}{\mathcal O}
\newcommand{\cE}{\mathcal{E}}

\newcommand{\cZ}{\mathcal{Z}}

\newcommand{\sZ}{\mathscr{Z}}

\begin{document}
\author{Yongyi Chen and Benjamin Howard}

\address{Department of Mathematics, Boston College, 140 Commonwealth Ave, Chestnut Hill, MA 02467, USA}
\email{howardbe@bc.edu}
\address{Department of Mathematics, Boston College, 140 Commonwealth Ave, Chestnut Hill, MA 02467, USA}
\email{chendgl@bc.edu}

\thanks{
B.H. was supported in part by NSF grant DMS-2101636.
}

\title{Intersection formulas on moduli spaces of unitary  shtukas}
\date{\today}

\begin{abstract}
Feng-Yun-Zhang have proved a function field analogue of the arithmetic Siegel-Weil formula, relating special cycles on moduli spaces of unitary shtukas to higher derivatives of Eisenstein series.
We prove an extension of this formula in a low-dimensional case, and deduce from it a Gross-Zagier style formula expressing intersection multiplicities of cycles in terms of higher derivatives of base-change $L$-functions.
\end{abstract}

\maketitle

\theoremstyle{plain}
\newtheorem{theorem}{Theorem}[subsection]
\newtheorem{bigtheorem}{Theorem}
\newtheorem{proposition}[theorem]{Proposition}
\newtheorem{lemma}[theorem]{Lemma}
\newtheorem{corollary}[theorem]{Corollary}
\newtheorem{conjecture}[theorem]{Conjecture}
\newtheorem{bigconjecture}[bigtheorem]{Conjecture}

\theoremstyle{definition}
\newtheorem{definition}[theorem]{Definition}
\newtheorem{hypothesis}[theorem]{Hypothesis}

\theoremstyle{remark}
\newtheorem{remark}[theorem]{Remark}
\newtheorem{example}[theorem]{Example}
\newtheorem{question}[theorem]{Question}

\renewcommand{\thebigtheorem}{\Alph{bigtheorem}}

\numberwithin{equation}{subsection}


\section{Introduction}


Feng-Yun-Zhang have defined special cycle classes on moduli spaces of unitary shtukas, and proved an arithmetic Siegel-Weil formula (in the sense of Kudla and Kudla-Rapoport) relating them to nonsingular Fourier coefficients of higher derivatives of  Siegel Eisenstein series.

One would like to  extend these results to singular coefficients, and then exploit this connection between arithmetic geometry and automorphic forms to prove a Gross-Zagier style formula relating intersection multiplicities of  special cycles to higher derivatives of Langlands $L$-functions.
Both problems are difficult because the moduli spaces on which the special cycles live are almost never proper.

The goal of this paper is to  prove an arithmetic Siegel-Weil formula for (some) singular Fourier coefficients in a low-dimensional case, and then use it to formulate and prove a Gross-Zagier style formula  in a way that circumvents the failure of properness of the moduli space of shtukas on which the special cycles live.  This Gross-Zagier style formula is conditional on the modularity conjecture of Feng-Yun-Zhang.


\subsection{The Feng-Yun-Zhang classes}


Let $k$ be a finite field of odd cardinality $q$, and fix a finite \'etale double cover
\[
X'\to X
\]
  of smooth, projective, geometrically connected   curves over  $k$.
 Denote by $\sigma \in \Aut(X'/X)$ the nontrivial Galois automorphism.

Fix  integers $n\ge 1$ and $r \ge 0$. Our \'etale double cover  determines a moduli stack  $\Sht^r_{\mathrm{U}(n)}$ of $\mathrm{U}(n)$-shtukas with $r$ legs.
  It is a Deligne-Mumford stack over $k$, locally of finite type, equipped with a smooth morphism
\[
\Sht^r_{\mathrm{U}(n)} \to (X')^r
\]
of relative dimension $r(n-1)$.  In particular, $\Sht^r_{\mathrm{U}(n)}$ is smooth over $k$ of dimension $rn$.
It is nonempty if and only if $r$ is even (\cite[Lemma 6.7]{FYZ1}), which we assume from now on.

To a pair $(\mathcal{E},a)$ consisting of  a vector bundle  $\mathcal{E}$ on $X'$ of rank $m \le n$ and a hermitian morphism $a\colon\mathcal{E} \to \sigma^* \mathcal{E}^\vee$ (in the sense of \S \ref{ss:notation}),  Feng-Yun-Zhang \cite[\S 7]{FYZ1} associate a \emph{naive special cycle} $\cZ^r_\mathcal{E}(a)$.
This is a Deligne-Mumford stack equipped with a finite  morphism
\[
\cZ^r_\mathcal{E}(a) \to \Sht^r_{\mathrm{U}(n)} .
\]
The naive special cycle has expected dimension $r(n-m)$, but this expectation is rarely fulfilled.
To correct for this, one finds in \cite[Definition 4.8]{FYZ2} the construction of a cycle class
\begin{equation}\label{FYZclass}
[ \cZ^r_\mathcal{E}(a) ] \in \mathrm{Ch}_{r(n-m)} ( \cZ^r_\mathcal{E}(a) )
\end{equation}
on $\cZ^r_\mathcal{E}(a)$ of the desired dimension.
When no confusion can arise, we denote in the same way the image of this class under the pushforward
\begin{equation}\label{bad push}
\mathrm{Ch}_{r(n-m)} ( \cZ^r_\mathcal{E}(a) )
\to
\mathrm{Ch}^{rm} ( \Sht^r_{\mathrm{U}(n)} ) .
\end{equation}
These special cycles are the function field analogues of the Kudla-Rapoport cycles \cite{KR}   on Shimura varieties for unitary groups of signature $(n-1,1)$.


It sometimes happens that the naive special cycle $ \cZ^r_\mathcal{E}(a)$ is a proper $k$-stack.
When this  is the case  the above pushforward can be refined to a homomorphism
\begin{equation}\label{good push}
\mathrm{Ch}_{r (n-m)} ( \cZ^r_\mathcal{E}(a) )
\to
\mathrm{Ch}_c^{rm} ( \Sht^r_{\mathrm{U}(n)} )
\end{equation}
valued in the Chow group with proper support, and  so \eqref{FYZclass}  determines  a class there as well.
  Using  the intersection pairing
 \begin{equation}\label{proper pairing}
\Ch^{r(n-m)}( \Sht^r_{\mathrm{U}(n)} ) \times \Ch_c^{rm}( \Sht^r_{\mathrm{U}(n)} )
\to
\Ch_c^{rn}( \Sht^r_{\mathrm{U}(n)} )
\end{equation}
and the degree map
\begin{equation}\label{degree}
\deg\colon  \mathrm{Ch}_c^{r n} (  \Sht_{\mathrm{U}(n)}^r ) \to \Q,
\end{equation}
we obtain a $\Q$-valued intersection  between $[\mathcal{Z}^r_\mathcal{E}(a)]$ and any codimension $r(n-m)$ cycle class on  $\Sht^r_{\mathrm{U}(n)}$.


\subsection{The modularity conjecture}


Let us restrict to the case $n=2m$, so that the special cycles \eqref{FYZclass} determine classes in the  middle codimension Chow group of $\Sht^r_{ \mathrm{U}(n) }$.

 Denote by $F=k(X)$ the field of rational functions on $X$, and similarly for  $F'=k(X')$.  Let
\[
\chi\colon \A_{F'}^\times \to \C^\times
\]
be an unramified Hecke character whose restriction to $\A_F^\times$ is trivial.

Let $\mathcal{A}(H_m)$ be the space of automorphic forms on the rank $2m$ quasi-split unitary group $H_m$ over $F$.
Any form  $f\in \mathcal{A}(H_m)$ fixed by the standard maximal compact open subgroup
 $K_m \subset H_m(\A_F)$  is determined by its \emph{geometric Fourier coefficients}
$f_{(\mathcal{E},a)}$.  These are defined in \S \ref{ss:fourier}, but at the moment all the reader needs to know is that this is a collection of complex numbers indexed by pairs $(\mathcal{E},a)$ consisting of a vector bundle $\mathcal{E}$ of rank $m$ on $X'$ and a hermitian morphism $a\colon\mathcal{E} \to\sigma^*\mathcal{E}^\vee$.

More generally, suppose  $C$ is any $\C$-vector space.  A function
\[
f : H_m(F)   \backslash H_m (\A_F) /K_m  \to C
\]
will be referred to as an \emph{unramified automorphic form valued in $C$}.    Any linear functional $\lambda : C \to \C$ determines a $\C$-valued automorphic form $\lambda \circ f$, and the \emph{geometric Fourier coefficients} of $f$ are the unique $f_{ ( \mathcal{E},a)} \in C$ satisfying
\[
\lambda( f_{ ( \mathcal{E},a)} )   = ( \lambda \circ f)_{ ( \mathcal{E},a)  }.
\]

The following is  the modularity conjecture of \cite[Conjecture 4.15]{FYZ2}, restricted to the case of cycles in middle codimension.
Abbreviate
\begin{equation}\label{eq:dE}
d(\mathcal{E}) \define \frac{\deg(\mathcal{E}^\vee) - \deg(\mathcal{E}) }{2} = \mathrm{rank}(\mathcal{E}) \deg(\omega_X) - \deg(\mathcal{E})
\end{equation}
for any vector bundle $\mathcal{E}$ on $X'$.

\begin{conjecture}[Feng-Yun-Zhang]\label{conj:modularity}
There is  an unramified automorphic form
\[
Z^{r,\chi}\colon  H_m(F)   \backslash H_m (\A_F) /K_m  \to \mathrm{Ch}^{r m}( \mathrm{Sht}^r_{U(n)}  )_\C
\]
 whose geometric Fourier coefficients   are the rescaled special cycles
\[
Z^{r,\chi}_{ ( \mathcal{E} ,a)  }   =
  \frac{ \chi( \det(\mathcal{E} ) ) } {  q^{ m d(\mathcal{E} ) } }  \,   [\cZ^r _{ \mathcal{E}}  ( a )  ] \in  \mathrm{Ch}^{r m}( \mathrm{Sht}^r_{U(n)}  )_\C .
\]
See \S \ref{ss:eisenstein} for the meaning of $\chi(\det(\mathcal{E}))$.
\end{conjecture}

Our first result provides some evidence toward this conjecture in the low-dimensional case of $m=1$ and $n=2$, in the same style as  \cite[Part 3]{FYZ2}.

\begin{theorem}\label{thm:simple duplication}
Suppose $\mathcal{E}_2$ is a line bundle on $X'$,  and $a_2\colon \mathcal{E}_2\iso \sigma^*\mathcal{E}_2^\vee$ is a hermitian isomorphism.  The naive special cycle
\[
\mathcal{Z}^r_{\mathcal{E}_2}(a_2) \to \Sht^r_{ \mathrm{U}(2) }
\]
 is proper over $k$, and
 there exists a $K_1$-fixed automorphic form $\mathscr{D} \in \mathcal{A}(H_1)$, depending on $(\mathcal{E}_2,a_2)$,  whose geometric Fourier coefficients are given by
\begin{equation}\label{intro duplication}
\mathscr{D}_{ (\mathcal{E}_1,a_1) } =
  \frac{ \chi( \mathcal{E}_1 ) } {  q^{  d(\mathcal{E}_1 ) } }  \,   \deg  \big(    [ \cZ^r_{\mathcal{E}_1}(a_1)]  \cdot [ \cZ^r_{\mathcal{E}_2}(a_2)] \big)
\end{equation}
for every line bundle $\mathcal{E}_1$ on $X'$ and every hermitian map $a_1:\mathcal{E}_1 \to \sigma^* \mathcal{E}_1^\vee$.

In other words, if we hold $(\mathcal{E}_2,a_2)$ fixed and allow $(\mathcal{E}_1,a_1)$ to vary, the intersection multiplicities on the right hand side of \eqref{intro duplication} are the geometric Fourier coefficients of an automorphic form.
\end{theorem}

This result appears in the text as part (3) of Theorem \ref{thm:final}.
The automorphic form $\mathscr{D}$  is essentially the kernel function appearing in the \emph{new way} integrals of \cite{PSRb, Qin}; see Remark \ref{rem:new way}.

The core of the proof of Theorem \ref{thm:simple duplication} is   Theorem  \ref{thm:intersection}.
This latter theorem generalizes  the Feng-Yun-Zhang arithmetic Siegel-Weil formula on $\Sht^r_{ \mathrm{U}(2) }$, which relates the degrees of certain $0$-cycles to nonsingular Fourier coefficients of Eisenstein series, to include at least some singular Fourier coefficients.


\subsection{Intersections with arithmetic theta lifts}


We continue to assume that $n=2m$ is even, and that
$
\chi : \A_{F'}^\times \to \C^\times
$
is an unramified Hecke character  whose restriction to $\A_F^\times$ is trivial.

Assuming Conjecture \ref{conj:modularity}, for a $K_m$-fixed  form  $f$ in an irreducible cuspidal automorphic representation $\pi \subset \mathcal{A}(H_m)$ we define  the \emph{arithmetic theta lift}
  \begin{equation}\label{arith theta}
\vartheta^{r,\chi}(f) = \int_{H_m(F) \backslash H_m(\A_F) } f(h) Z^{r,\chi}(h) \, dh \in  \Ch^{rm}( \Sht^r_{\mathrm{U}(n)} ) _\C.
  \end{equation}
Following ideas of Kudla \cite[\S 8]{Kud} in the context of  Shimura varieties, developed further by Liu \cite{Liu} and Li-Liu \cite{LL},
one could  hope for an \emph{arithmetic Rallis inner product formula}
\begin{equation}\label{RIP}
\deg  \big(  \vartheta^{r,\chi}(f) \cdot  \vartheta^{r,\chi}(f)   \big)
\stackrel{?}{=}  C\cdot \frac{d^r}{ds^r} \Big|_{s=0}  L(s+1/2, \mathrm{BC}(\pi) \otimes \chi)
\end{equation}
(for some explicit constant $C$)
relating the self-intersection multiplicity of $\vartheta^{r,\chi}(f)$ to  the $r^\mathrm{th}$ central derivative of the twisted base-change $L$-function at the center  of its functional equation.

Unfortunately, in the present context the left hand side of \eqref{RIP} is not defined, as  the moduli space $\Sht^r_{\mathrm{U}(n)}$ is not proper as soon as $n \ge 2$.
While one can form the self-intersection of $ \vartheta^{r,\chi}(f)$ as an element in the usual Chow group (i.e.~not the Chow group with proper support) of $0$-cycles on $\Sht^r_{\mathrm{U}(n)}$,  there is no analogue of the degree map \eqref{degree} on this Chow group.

To circumvent this, fix a pair $(\mathcal{E}_2,a_2)$ consisting of a rank $m$ vector bundle $\mathcal{E}_2$ on $X'$ and a hermitian morphism
$a_2 : \mathcal{E}_2 \to \sigma^* \mathcal{E}_2^\vee$, and  \emph{assume} that the associated naive special cycle
$
\cZ^r_{\mathcal{E}_2}(a_2) \to \Sht^r_{\mathrm{U}(n)}
$
 is proper over $k$.    Using  \eqref{proper pairing} and \eqref{degree} to  intersect
\[
[ \cZ^r_{\mathcal{E}_2}(a_2)]  \in  \Ch^{rm}_c( \Sht^r_{ \mathrm{U}(n) } )
\]
against the arithmetic theta lift \eqref{arith theta}, we propose the following Gross-Zagier style intersection formula.

\begin{conjecture}\label{conj:GKZ}
Assuming the properness of $\cZ^r_{\mathcal{E}_2}(a_2)$,  for any $K_m$-invariant   $f\in \pi \subset \mathcal{A}(H_m)$ as above  we have
\begin{align}    \label{GKZ}
& \deg \big(   \vartheta^{r,\chi} (f) \cdot [ \cZ^r_{\mathcal{E}_2}(a_2)] \big)  \\
& =
  f_{( \mathcal{E}_2 , -a_2)  }    \frac{    q^{md(\mathcal{E}_2) }    }{ (\log q)^r}
    \frac{d^r}{ds^r}\Big|_{s=0}\left( q^{  ns \deg ( \omega_X)  }    L( s +1/2  ,  \mathrm{BC}(\pi) \otimes \chi )\right),
    \nonumber
\end{align}
where $f_{(\mathcal{E}_2,-a_2)}$ is the geometric Fourier coefficient   in the sense of \S \ref{ss:fourier}.
\end{conjecture}

The properness assumption imposed  on  $\cZ^r_{\mathcal{E}_2}(a_2)$ is very strong.
Although no proof is provided, it is claimed in  \cite[Example 4.20]{FYZ2} that it holds
whenever  $m=1$ and $a_2 : \mathcal{E}_2 \to \sigma^* \mathcal{E}_2^\vee$ is injective.
 There is circumstantial evidence that properness may also hold for some middle codimension special cycles on $\Sht^r_{\mathrm{U}(4)}$.  For $m>2$ there is no expectation that any  middle codimension special cycles on $\Sht^r_{\mathrm{U}(2m)}$ are proper.
 The point we wish to convey is that the above conjecture predicts that \emph{every} instance of a proper special cycle in middle dimension has its own  higher derivative  intersection formula, and so  even these low-dimensional cases are of interest.

\begin{remark}
Conjecture \ref{conj:GKZ} is perhaps less like the Gross-Zagier theorem and more like a hybrid of the arithmetic Rallis inner product formula and the Gross-Kohnen-Zagier theorem.
An analogue of it on quaternionic Shimura curves over $\Q$ can be deduced easily from \cite[Corollary 1.0.7]{KRY}.
\end{remark}

Our  second main result  is the proof of the above conjecture  in the simplest case, stated as part (4) of Theorem \ref{thm:final} in the body of the text.
This result is conditional on Conjecture \ref{conj:modularity}, which is needed to even define the arithmetic theta lift \eqref{arith theta}.

\begin{theorem}\label{thm:GKZ}
Suppose $\mathcal{E}_2$ is a line bundle on $X'$ (so $m=1$ and $n=2$) and $a_2\colon \mathcal{E}_2\iso \sigma^*\mathcal{E}_2^\vee$ is a hermitian isomorphism.  The naive special cycle
$
\cZ^r_{\mathcal{E}_2}(a_2) \to \Sht^r_{\mathrm{U}(2)}
$
is proper over $k$, and  the equality \eqref{GKZ} holds.
\end{theorem}

Proving  Theorems \ref{thm:simple duplication} and \ref{thm:GKZ} under the weaker hypothesis that the hermitian morphism $a_2\colon \mathcal{E}_2 \to \sigma^*\mathcal{E}_2^\vee$ is injective might be within reach.
  Even more interesting would be to find examples of proper special cycles of middle codimension on $\Sht^r_{\mathrm{U}(4)}$, and extend our results to that setting.


\subsection{Notation}
\label{ss:notation}


The  double cover $X'\to X$ over $k$ fixed above remains fixed through the paper,  $\sigma \in \Aut(X'/X)$ is its nontrivial automorphism, and the nontrivial automorphism of $F'=k(X')$ over $F=k(X)$ is denoted the same way.
Denote by
 \begin{equation}\label{quadratic character}
\eta =\eta_{F'/F}\colon \A_F^\times \to \{\pm 1\}
\end{equation}
 the associated  quadratic character.

Let $\omega_X$ be the sheaf of K\"ahler differentials on $X$.
By the assumption that $X'\to X$ is \'etale, its pullback is the sheaf of K\"ahler differentials  $\omega_{X'}$ on $X'$.
Denote by  $\omega_F$ the stalk of $\omega_X$ at the generic point of $X$.  In other words, $\omega_F$ is the $1$-dimensional $F$ vector space of rational K\"ahler differentials on $X$.

 As in \cite{FYZ1}, for any vector bundle $\mathcal{E}$ on $X'$ we denote by
\[
\mathcal{E}^\vee =  \underline{\Hom}_{\co_{X'}} ( \mathcal{E} , \omega_{X'})
\]
 the \emph{Serre dual} of $\mathcal{E}$.
 A \emph{hermitian morphism} $a\colon\mathcal{E} \to \sigma^* \mathcal{E}^\vee$  is  a morphism (in the category of coherent sheaves on $X'$) satisfying the hermitian condition $\sigma^* a^\vee = a$.

If $R$ is an $F$-algebra we abbreviate $R'=R\otimes_F F'$.
If  $L$ is a one-dimensional $F$-vector space  (e.g.~$\omega_F$),  denote by
\[
\Herm_n(R,L) \subset M_n( L \otimes_F R')
\]
 the $R$-submodule of  hermitian matrices.  In particular,
 \[
 \Herm_n(R) \define \Herm_n(R,F)
 \]
  is the usual $R$-module of $n\times n$ hermitian matrices with entries in $R'$.

For any group scheme $G$ over $F$ we  abbreviate $[G] = G(F) \backslash G(\A_F)$.  In particular, this applies to the scheme
  $\Herm_n$ over $F$ with functor of points $R\mapsto \Herm_n(R)$.


\subsection{Acknowledgements}


Yongyi Chen thanks Keerthi Madapusi and Tony Feng for helpful conversations.

Ben Howard thanks Jan Bruinier, who  suggested (some years ago) that  the integral representations of  \cite{PSRb} might be used  to obtain interesting Gross-Zagier style intersection formulas. This suggestion was the inspiration for the present work.
He also thanks Tony Feng, Aaron Pollack, Chris Skinner, and Wei Zhang for helpful conversations during the SLMath/MSRI semester  \emph{Algebraic cycles, $L$-values, and Euler systems}, where part of this work was carried out.


\subsection{Statements and declarations}


The authors have no competing interests to declare.
The preparation of this manuscript did not involve the generation or analysis of any datasets.


\section{Eisenstein series on  unitary groups}
\label{s:automorphic}


This section  begins by establishing some basic notation for quasi-split unitary groups, and the Siegel Eisenstein series on them.
Once that is done we prove two results.

The first, Theorem \ref{thm:duplication}  is a modest generalization of the famous doubling formula of  Piatetski-Shapiro and Rallis \cite{PSRa},  providing the link between Siegel Eisenstein series and base-change $L$-functions.  This will be needed in the proof of Theorem \ref{thm:final}.

The second, Proposition \ref{prop:genus drop},  is  an explicit formula relating singular Fourier coefficients of an Eisenstein series on a unitary group  of rank $4$ to non-singular Fourier coefficients of an Eisenstein series on a unitary group  of  rank $2$.
This is one of the main ingredients in the proof of Theorem \ref{thm:intersection}.


\subsection{The quasi-split unitary group}
\label{ss:thegroup}


For  an integer $n\ge 1$, abbreviate
\[
w_n = \begin{pmatrix}   & I_n \\ -I_n & \end{pmatrix} \in \GL_{2n}(F).
\]
We  endow the space of column vectors  $W_n=(F')^{2n}$ with the
standard skew-hermitian form $h_n(x,y) = {}^tx  \cdot  w_n \cdot    \sigma(y)$.
The associated rank $2n$ quasi-split unitary group over $F$ is denoted
\[
H_n =  \mathrm{U}(W_n) \subset \mathrm{Res}_{F'/F} \GL_{2n}.
\]

The  \emph{standard Siegel parabolic}
$
P_n \subset H_n
$
is the subgroup of matrices whose lower left $n\times n$ block vanishes.
Its unipotent radical is denoted
\[
N_n  = \left\{
n(b) = \begin{pmatrix}
I_n & b \\ & I_n
\end{pmatrix}  : b \in \Herm_n
\right\}
\]
while its Levi factor is denoted
\[
M_n = \left\{
m(\alpha) = \begin{pmatrix}
\alpha  &  \\ & \sigma( {}^t \alpha^{-1} )
\end{pmatrix}  : \alpha \in \mathrm{Res}_{F'/F}\GL_n
\right\} .
\]
The  \emph{standard compact open subgroup}  of $H_n(\A_F)$ is
$K_n = \prod_v K_{n,v}$, where for any place $v$ of $F$ we set
\[
K_{n,v} = H_n(F_v) \cap \GL_{2n}(\co_{F'_v}).
\]

\begin{remark}\label{rem:measures}
The Haar measure on $H_n(\A_F)$ is always normalized so that the standard compact open $K_n$ has volume $1$, and
$[H_n]$ is given the induced quotient measure.
We give  $\Herm_n(\A_F) \iso N(\A_F)$  the self-dual Haar measure; explicitly, this is the one for which the compact open subgroup
\[
\prod_v \Herm_n( \co_{F_v}   ) \subset \Herm_n(\A_F)
\]
 of hermitian matrices with  integral entries has volume
$q^{-  n^2  \deg(\omega_X) /2 } $.  The induced quotient measure on $[\Herm_n] \iso [N]$  has volume $1$.
\end{remark}


\subsection{Fourier coefficients}
\label{ss:fourier}


Fix   a nontrivial additive character $\psi_0\colon k \to \C^\times$.   As in \cite[\S 2.2]{FYZ1},
we use the $k$-linear \emph{residue map}
$
\mathrm{Res}\colon\omega_F \to k
$
to define  a canonical pairing
\[
\langle \cdot , \cdot \rangle\colon \Herm_n(\A_F, \omega_F) \times \Herm_n(\A_F) \to k
\]
 by $\langle T,b\rangle = \mathrm{Res}( - \mathrm{Tr}(Tb))$.  Its composition with $\psi_0$ is denoted
\[
\langle \cdot , \cdot \rangle_{\psi_0}\colon\Herm_n(\A_F, \omega_F) \times \Herm_n(\A_F) \to \C^\times .
\]

  Any automorphic form $f\in \mathcal{A}(H_n)$ has a Fourier expansion
\begin{equation}\label{general fourier}
f(g) = \sum_{ T\in \Herm_n(F,\omega_F) } f_T(g) ,
\end{equation}
in which the $T$-coefficient is
\[
f_T(g)  = \int_{[  \Herm_n ]  }
f ( n(b ) g )  \,  \langle T,b\rangle_{\psi_0}   \, db,
\]
with Haar measure as in Remark \ref{rem:measures}.

Suppose $f \in \mathcal{A}(H_n)$ is right invariant under $K_n$.
By the Iwasawa decomposition, $f$ is determined by its restriction to $P_n(\A_F) \subset H_n(\A_F)$, and hence is determined by the collection of functions
$
\{ f_T|_{M_n(\A_F)} \}_{ T\in \Herm_n(F , \omega_F)}.
$
As explained in  \cite[\S 2.6]{FYZ1},  this gives rise to a collection of \emph{geometric Fourier coefficients of $f$}:
 complex numbers $f_{ ( \mathcal{E} ,a)} $ indexed by pairs $(\mathcal{E},a)$ consisting of a rank $n$ vector bundle $\mathcal{E}$ on $X'$ and a hermitian morphism  $a\colon \mathcal{E}  \to \sigma^*\mathcal{E}^\vee.$

The precise definition of $f_{ ( \mathcal{E},a) }$ is as follows.
The generic fiber of $\mathcal{E}$ is a free $F'$-module of rank $n$, and upon choosing an isomorphism
$
 \mathcal{E}_{F'} \iso (F')^n
$
the hermitian morphism   $a$ is given by a hermitian matrix $T \in  \Herm_n(F , \omega_F)$.
For every place $v$ of $F'$ there is an $\alpha_v \in \GL_n(F'_v)$ such that the completed stalk of $\mathcal{E}$ at $v$ is identified with the lattice $\alpha_v \cdot \co_{F'_v}^n \subset (F_v')^n$.
Setting  $\alpha=\prod_v \alpha_v \in \GL_n(\A_{F'})$,  we  define
\[
f_{ ( \mathcal{E} ,a)} = f_T( m(\alpha)).
\]

\begin{remark}\label{rem:vb identification}
We identify the set of rank $n$ vector bundles on $X'$ with the double quotient
\[
\GL_n(F') \backslash \GL_n(\A_{F'})  /  \prod_v  \GL_n ( \co_{F'_n} )
\]
by sending $\mathcal{E} \mapsto \alpha$ with $\alpha$ as above.  In particular,
\[
q^{ \deg( \det(\mathcal{E})) } = | \det(\alpha) |_{F'} .
\]
\end{remark}


\subsection{Siegel Eisenstein series}
\label{ss:eisenstein}


Fix an unramified   Hecke character  $\chi\colon  \A_{F'} ^\times \to \C^\times$, and denote by $\chi_0$ its restriction to $\A_F^\times$.
We always assume
\begin{equation}\label{hecke restriction}
\chi_0= 1 \quad \mbox{or} \quad \chi_0 = \eta ,
\end{equation}
where $\eta$ is the quadratic character  \eqref{quadratic character}.
We often use the isomorphism
\[
\mathrm{Pic}(X') \iso  (F')^\times \backslash \A_{F'}^\times / \prod_v \co_{F'_v}^\times
\]
of Remark \ref{rem:vb identification} to view  $\chi$ as a character
$
\chi\colon \mathrm{Pic}(X') \to \C^\times.
$

\begin{remark}\label{rem:hecke confusion}
We will follow the discussion of \cite{FYZ1} for Eisenstein series on $H_n(\A_F)$, but warn the reader that
 throughout \emph{loc.~cit.} it is assumed that the Hecke character $\chi$ satisfies $\chi_0= \eta^n$.
 Thus some of the formulas of \emph{loc.~cit.} require small modifications to account for our weaker assumption  \eqref{hecke restriction}.
This weaker assumption   is the only one imposed in \cite{Tan}, and everything we need to know about Eisenstein series can be deduced from the results found there.
\end{remark}

For a complex variable $s$, denote by
\[
 I_{n}(s,\chi) =  \mathrm{Ind}_{P_{n} (\A_F)}^{H_{n}(\A_F)}(\chi |\cdot|_{F'} ^{s+\frac{n}{2}} )
\]
the unnormalized smooth  induction.
Here we are viewing both $\chi$ and $|\cdot|_{F'} $ as  characters of $P_{n}(\A_F)$ by composing them with the homomorphism
 \begin{equation}\label{parabolic character}
P_{n} (\A_F) \to \A_{F'}^\times
\end{equation}
sending an element of $P_{n} (\A_F) \subset \GL_{2n}(\A_{F'})$ to  the determinant of its upper left $n\times n$ block.

Given any $\Phi_s  \in I_n(s,\chi)$ with $\mathrm{Re}(s) > n/2$, the summation
\[
E (g , \Phi_s ) = \sum_{ \gamma \in P_{n}(F) \backslash H_{n}(F) } \Phi_s( \gamma g  )
\]
is absolutely convergent  and  defines a Siegel Eisenstein series on $H_{n}(\A_F)$.
Suppose $T\in \Herm_n(F,\omega_F)$ is nonsingular, in the sense that  $\det(T)\neq 0$ after fixing some trivialization $\omega_F\iso F$.
As in \cite[\S 2.2]{FYZ1},  modified to account for the choice of Haar measure in Remark \ref{rem:measures},  the  corresponding Fourier coefficient is
\begin{equation}\label{whit factor}
E_T(g,\Phi_s) =  \int_{  \Herm_n( \A_F) }
\Phi_s \big( w_n^{-1} n(b ) g  \big)  \,   \langle T,b \rangle_{\psi_0}  \, db .
\end{equation}

There is a unique standard section $s\mapsto \Phi_s^\circ \in I_n(s,\chi)$  satisfying  $\Phi_s^\circ(k)=1$ for all $k\in K_n$.  We call this the \emph{normalized spherical section}, and denote the associated  \emph{spherical Eisenstein series}  by
\[
E(g, s,\chi) = E(g,\Phi_s^\circ).
\]

The standard $L$-function  $L(s,\eta^i \chi_0 )$  satisfies the functional equation
\begin{equation}\label{L functional}
q^{ \frac{s}{2} \deg(\omega_X) }  L(s,\eta^i \chi_0 ) =   q^{ \frac{1-s}{2} \deg(\omega_X) }    L(1-s,\eta^i \chi_0 ).
\end{equation}
The product of $L$-functions
\[
\mathscr{L}_{n}(s,\chi_0) =  \prod_{i=1}^n  L(2s+i, \eta^{i-n} \chi_0 )
\]
 agrees with the $\mathscr{L}_n(s)$ defined in \cite[\S 2.6]{FYZ1} when $\chi_0 = \eta^n$.

Because of the assumption \eqref{hecke restriction},  there is an intertwining operator
\[
M_n(s)\colon I_n(s,\chi) \to I_n(-s,\chi)
\]
 defined for  $\mathrm{Re}(s) > n/2$ by
\begin{equation}\label{intertwining}
(M_n (s)\Phi_s) ( g) =    \int_{  \Herm_n(\A_F)   }
  \Phi_{s} \left( w_n^{-1}   n(b)    g  \right)  \, d b .
  \end{equation}
By combining the  local calculation of \cite[Proposition 2.1]{Tan}, as simplified on page 170 of \emph{loc.~cit.},  with the global function equation \eqref{L functional}, one finds that  the unramified spherical section satisfies
\begin{equation}\label{spherical intertwining}
  M_n(s) \Phi_s^\circ =    q^{     -  2ns  \deg(\omega_X)   }
     \frac{  \mathscr{L}_{n}(-s,\chi_0) }  { \mathscr{L}_{n}(s,\chi_0)   }   \cdot  \Phi^\circ_{-s} .
\end{equation}
It follows that $E(g,s,\chi)$   has meromorphic continuation, and that the renormalized Eisenstein series
\begin{equation}\label{renormalized eisenstein}
\widetilde{E}(g,s,\chi) \define  q^{  ns \deg ( \omega_X)  }    \cdot    \mathscr{L}_{n}(s,\chi_0)  \cdot  E ( g ,s,\chi )
\end{equation}
of \cite[\S 9.5]{FYZ2} satisfies the functional equation
\begin{equation}\label{eis functional}
\widetilde{E}(g,s,\chi)=\widetilde{E}(g,-s,\chi).
\end{equation}

For nonsingular $T$, the integral \eqref{whit factor}  determined by the normalized spherical section $\Phi_s^\circ$ can be expressed as products of local representation densities.  Hence the same is true of the geometric Fourier coefficient $E_{( \mathcal{E},a)}(s,\chi)$ indexed by a rank $n$ vector bundle $\mathcal{E}$ on $X'$ and an \emph{injective} hermitian morphism $a\colon \mathcal{E} \to \sigma^* \mathcal{E}^\vee$.  We only need the simplest case of these formulas, in which $a$ is an isomorphism.

\begin{proposition}\label{prop:unr eisenstein coefficient}
If  $\mathcal{E}$ is a rank $n$ vector bundle on $X'$, and $a\colon\mathcal{E} \iso \sigma^* \mathcal{E}^\vee$ is a hermitian isomorphism, then
\[
 q^{  ns \deg ( \omega_X)  }    \cdot    \mathscr{L}_{n}(s, \chi_0)  \cdot  E_{( \mathcal{E},a)} ( s,\chi )
 = \chi(\det(\mathcal{E})).
\]
\end{proposition}

\begin{proof}
This is a special case of \cite[Theorem 2.8]{FYZ1}, slightly extended
(see Remark \ref{rem:hecke confusion})  using the  calculation of unramified local Whittaker functions found in \cite[Proposition 3.2]{Tan}.
\end{proof}


\subsection{The duplication formula}


In this subsection we  assume that $n=2m$ is even.
We continue to work with an unramified   Hecke character  $\chi\colon    \A_{F'} ^\times \to \C^\times$, but now assume $\chi_0 = 1$ instead of the weaker \eqref{hecke restriction}.

Consider the Eisenstein series
\[
E (g , s ,\chi ) = \sum_{ \gamma \in P_{n}(F) \backslash H_{n}(F) } \Phi^\circ_s( \gamma g  )
\]
 on $H_{n}(\A_F)$ associated to the normalized spherical section $\Phi_s^\circ \in I_{n}(s,\chi)$.
 We are interested in its pullback via the  \emph{standard embedding}
$
i_0\colon H_m \times H_m \to H_{n}
$
defined by
\begin{equation}\label{standard embedding}
i_0 \left(
\begin{pmatrix}  a_1 & b_1 \\ c_1 & d_1 \end{pmatrix} ,
\begin{pmatrix}  a_2 & b_2 \\ c_2 & d_2 \end{pmatrix}
\right)
=
\begin{pmatrix}
a_1 &   & b_1  &   \\
 & a_2 &   & b_2   \\
 c_1  &   & d_1 &    \\
    &  c_2 &   &  d_2
    \end{pmatrix}.
\end{equation}

\begin{definition}
 The \emph{doubling kernel} is  the  two-variable automorphic form
\begin{equation}\label{Dkernel}
D( g_1, g_2 , s,\chi )  =  E(i_0 (g_1,g_2), s,\chi )  \in \mathcal{A}( H_m\times H_m).
\end{equation}
For any $T_2 \in \Herm_m(F,\omega_F)$ define  the \emph{new way kernel}
\begin{equation}\label{new way}
 D_{\square, T_2}(g_1,g_2,s ,\chi )    =    \int_{ [ \Herm_m ] }   D (  g_1, n(b)  g_2   ,s ,\chi)
   \, \langle T_2 , b \rangle_{\psi_0}   \, d b
\end{equation}
 as  the $T_2$-coefficient of the doubling kernel with respect to the variable $g_2$.
\end{definition}

As in \eqref{general fourier},  any  automorphic form $f \in  \mathcal{A}( H_m \times H_m )$ in two variables has a double Fourier expansion
\[
f( g_1,g_2 )= \sum_{ T_1 , T_2 \in \Herm_m (F,\omega_F) }  f_{T_1,T_2}(g_1,g_2).
\]
The double Fourier coefficients of the doubling kernel \eqref{Dkernel}
are related to the Fourier coefficients of $E(g, s,\chi )$ by
\[
D_{T_1,T_2}(g_1, g_2 , s,\chi)  =
\sum_{
T = \left(\begin{smallmatrix} T_1 & * \\ * & T_2 \end{smallmatrix} \right) \in \Herm_{n}(F,\omega_F)
 }
 E_T( i_0( g_1,g_2)  , s ,\chi ),
\]
and so the new way kernel \eqref{new way} has   Fourier expansion
\begin{align*}
 D_{\square, T_2}(g_1,g_2,s ,\chi )     &   =
  \sum_{ T_1   \in \Herm_m (F) }  D_{T_1,T_2}(g_1,g_2,s  ,\chi)    \\
&   =
 \sum_{
T = \left(\begin{smallmatrix} *  & * \\ * & T_2 \end{smallmatrix} \right) \in \Herm_{n}(F,\omega_F)
 }
 E_T( i_0( g_1,g_2)  , s ,\chi ).
 \end{align*}

 As in \S \ref{ss:fourier}, we can take geometric Fourier coefficients of the  doubling kernel  \eqref{Dkernel}  in the second variable.
 The result is, for every pair $(\mathcal{E}_2,a_2)$ consisting of   a rank $m$ vector bundle $\mathcal{E}_2$ on $X'$ and a hermitian map $a_2\colon\mathcal{E}_2 \to \sigma^* \mathcal{E}_2^\vee$,  an automorphic form
 \begin{equation}\label{geometric new way kernel}
 D_{\square, ( \mathcal{E}_2,a_2)  }(g_1 ,s,\chi) \in \mathcal{A}(H_m)
 \end{equation}
 in the variable $g_1$ whose geometric Fourier coefficients are  related to those of the Eisenstein series  $E(g,s,\chi)$ by
 \begin{equation}\label{eq:geometric doubling coefficients}
D_{  (\mathcal{E}_1,a_1) , ( \mathcal{E}_2,a_2) }  ( s,\chi)
=
\sum_{  a = \left( \begin{smallmatrix} a_1 & * \\ * & a_2 \end{smallmatrix} \right) }
E_{ (\mathcal{E} , a ) } ( s , \chi).
\end{equation}
On the right hand side $\mathcal{E} = \mathcal{E}_1\oplus \mathcal{E}_2$, and
 the sum is over all hermitian maps $a\colon\mathcal{E} \to \sigma^*\mathcal{E}^\vee $ for which  the composition
\[
 \mathcal{E}_i  \map{\mathrm{inc.}} \mathcal{E}_1 \oplus \mathcal{E}_2 \map{a}  \sigma^*\mathcal{E}_1^\vee \oplus  \sigma^*\mathcal{E}_2^\vee \map{\mathrm{proj.}}   \sigma^*\mathcal{E}_i^\vee
\]
 agrees with $a_i$ for both $i \in \{1,2\}$.

\begin{theorem}[The duplication formula]\label{thm:duplication}
If  $f\in \pi$ is a $K_m$-fixed vector in an irreducible cuspidal automorphic representation
$\pi \subset \mathcal{A}(H_m)$, then for all $g_2 \in H_m(\A_F)$ we have
\begin{align*}
  \int_{ [H_m]   } D( g_1 , g_2   , s ,\chi )   f(g_1) \, d g_1
 =   \frac{  L( s + 1/2  ,\mathrm{BC}(\pi) \otimes \chi  )   }{\mathscr{L}_{n}(s, \chi_0)   }
 \cdot   \chi(\det( g_2 ))  \cdot f(g_2^\dagger) .
\end{align*}
Here  $g\mapsto g^\dagger$ is the automorphism of  $H_m(\A_F)$  defined on block matrices by
\[
\begin{pmatrix}  a  & b  \\ c & d \end{pmatrix}^\dagger
=
\begin{pmatrix}  a  & -b  \\ -c & d \end{pmatrix},
\]
where each block $a,b,c,d$ has size $m\times m$.
The $L$-function on the right hand side is the twisted base-change $L$-function as defined in \cite{Clo}. See also \cite{Min}.
\end{theorem}

\begin{remark}\label{rem:almost doubling}
In the equality of Theorem \ref{thm:duplication}, if one takes the Petersson inner product of both sides against any cusp form  $\breve{f} \in \mathcal{A}(H_m)$ the resulting formula is the usual \emph{doubling formula} of Piatetski-Shapiro and Rallis \cite{PSRa}, extended to the unitary case by Li \cite{Li}; see also \cite{Liu}.
By varying the auxiliary form $\breve{f}$, the  duplication formula above would follow  immediately from the doubling formula if one knew a priori that the integral on the left hand side of Theorem \ref{thm:duplication} defined a \emph{cuspidal} automorphic form in the variable $g_2$.  Thus the only real new information in Theorem \ref{thm:duplication} is the cuspidality of the left hand side.
\end{remark}

Before giving the proof, we restate Theorem \ref{thm:duplication}  in the precise form in which we will use it.
If  the automorphic form $f (g) $ in Theorem \ref{thm:duplication} has  $T^\mathrm{th}$ Fourier coefficient $f_T$, then the automorphic form $f(g^\dagger)$ has $T^\mathrm{th}$ Fourier coefficient $f_{-T} (g^\dagger)$.  Thus taking the  $T_2$-coefficient in both sides of Theorem \ref{thm:duplication} yields the formula
 \begin{align}\label{new way integral}
 &  \int_{  [H_m]  }  D_{\square,T_2}( g_1, g_2   ,s ,\chi  )   f(g_1) \, d g_1  \\
&   =   \frac{  L( s +1/2  , \mathrm{BC}(\pi)  \otimes \chi )   }{  \mathscr{L}_{n}(s,\chi_0)    }
    \cdot   \chi(\det( g_2 ))  \cdot  f_{-T_2} (g_2^\dagger) . \nonumber
\end{align}
Equivalently, expressed in  the language of geometric Fourier coefficients,
 \begin{align}\label{geometric new way}
  &\int_{ [H_m ] }  D_{\square, ( \mathcal{E}_2,a_2) }( g_1   ,s,\chi )   f(g_1 ) \, d g_1  \\
    & =
     \frac{  L( s +1/2  , \mathrm{BC}(\pi)  \otimes \chi )   }{  \mathscr{L}_{n}(s,\chi_0)    }
   \cdot  \chi(\det(\mathcal{E}_2))    \cdot   f_{( \mathcal{E}_2 , -a_2)  }    .
          \nonumber
\end{align}

\begin{remark}\label{rem:new way}
Extending ideas of  B\"ocherer \cite{Bo1,Bo2} from symplectic groups to unitary groups,
Ikeda \cite{Ike} showed that the new way kernel \eqref{new way}  is a linear combination of Eisenstein series and theta series for the hermitian space determined by  $T_2$.
As explained in  \cite{GiSo},  if one substitutes this linear combination into the integral on the left hand side of \eqref{new way integral}  and sets $g_2=1$,  the resulting formula is essentially the \emph{new way integral}  of \cite{PSRb,Qin}.
This explains our choice of terminology for \eqref{new way}.
\end{remark}

Now we turn to the proof of Theorem \ref{thm:duplication}.
In addition to the standard embedding   \eqref{standard embedding}, we will make use of the twisted embedding
\begin{equation}\label{twisted embedding}
 i\colon H_m \times H_m \to H_n
\end{equation}
defined by
$i (g_1,g_2) =  i_0(g_1,g_2^\dagger)$.  Set
\[
\delta = \begin{pmatrix}
 & &  I_m &  \\
 &  I_m & &   \\
 - I_m & I_m  & &    \\
 &   & I_m &  I_m
\end{pmatrix} \in H_{n}(F) \cap K_{n} \subset \GL_{2n}(F')  .
\]

\begin{lemma}\label{lem:true double}
Consider the doubled space $W_n^\Delta=W_m\oplus W_m$ endowed with the skew-hermitian form $h_n^\Delta=h_m \oplus -h_m$.
There is an isomorphism
\[
H_n^\Delta \define U(W_n^\Delta)  \iso H_n
\]
identifying   the canonical  inclusion
$
 i^\Delta\colon H_m\times H_m \to   H_n^\Delta
$
with the twisted embedding \eqref{twisted embedding},
 and identifying   the stabilizer  $P_n^\Delta \subset H_n^\Delta$ of the diagonal Lagrangian
$
\{ (x,x)  : x\in W_m \} \subset W_n^\Delta
$
  with  the conjugate $\delta^{-1} P_n \delta \subset H_n$ of the  standard Siegel parabolic.
  \end{lemma}

\begin{proof}
Using the standard basis  $e_1,\ldots, e_m, f_1,\ldots, f_m \in (F')^{2m} = W_m$, we define a basis
 $e^\Delta_1,\ldots, e^\Delta_n,f^\Delta_1,\ldots, f^\Delta_n \in W_n^\Delta$  by
\begin{align*}
e^\Delta_i  = \begin{cases}
(e_i,0) & \mbox{if } 1\le i \le m \\
(0,e_{i-m}) & \mbox{if } m<i \le n
\end{cases}  \quad \mbox{and}\quad
f^\Delta_i  = \begin{cases}
(f_i,0) & \mbox{if } 1\le i \le m \\
(0,-f_{i-m}) & \mbox{if } m<i \le n .
\end{cases}
\end{align*}
The  $F'$-linear isomorphism $W_n^\Delta  \iso (F')^{2n} = W_n$ determined by this basis  is  an isometry, and the induced isomorphism
of unitary groups has the desired properties.
\end{proof}

\begin{proposition}\label{prop:duplication unfolding}
For any cuspidal automorphic form $f \in \mathcal{A}(H_m)$ and any standard section $\Phi_s \in I_n(s,\chi)$ we have the equality
\[
  \int_{ [H_m] } E( i  (g_1,g_2)  , \Phi_s )   f(g_1) \, d g_1  =   \int_{  H_m (\A_F)  } \Phi_s(   \delta i (   g_1 , g_2 )  ) f(g_1 ) \, d g_1
\]
for all $g_2\in H_m(\A_F)$.
\end{proposition}

\begin{proof}
The proof  follows a similar line of reasoning as in  \cite{PSRa}, where it is shown that the two sides of the desired equality agree after taking the Petersson inner product against any cusp form in $\mathcal{A}(H_m)$.
See Remark \ref{rem:almost doubling}.

By mild abuse of notation, we denote by
\begin{equation}\label{twisted abuse}
H_m \times H_m' \subset H_n
\end{equation}
the image of the twisted embedding \eqref{twisted embedding}.
In other words, abbreviate $H_m=i(H_m,1)$ and $H_m' = i(1,H_m)$.
Using Lemma \ref{lem:true double}, one sees that  \eqref{twisted abuse}   is the stabilizer of an orthogonal splitting
\[
W_n= V \oplus V'
\]
 with $V$ isometric to $(W_m,h_m)$ and $V'$ isometric to $(W_m,-h_m)$.

Inserting the definition of the Eisenstein series and unfolding shows that
\begin{align} \label{duplicate unfolding}
&  \int_{ [H_m]   } E( i  (g_1,g_2)   ,\Phi_s ) \,  f(g_1) \, d g_1  \\
& =      \sum_{  \gamma \in P_n(F) \backslash H_n(F) / H_m(F)  }
\int_{ I_\gamma (F) \backslash H_m(\A_F)  } \Phi_s( \gamma i ( g_1,g_2)   ) \,  f(g_1) \, d g_1 \nonumber
\end{align}
where $I_\gamma = H_m \cap   \gamma^{-1} P_n \gamma$.
It follows from Lemma \ref{lem:true double} that $I_\delta=1$, and we are reduced to proving the vanishing of the integrals indexed by all $\gamma \neq \delta$.

In practice, this amounts to upgrading the analysis of right $(H_m \times H_m')$-orbits in $P_n\backslash H_n$ found in \cite[\S 2]{PSRa} to an analysis of the right $H_m$-orbits.
Say that a subgroup $I \subset H_m$ is \emph{negligible} if there exists a proper parabolic subgroup $Q \subset H_m$ with unipotent radical $N$ such that $N  \subset I  \subset Q.$

\begin{lemma}
Consider the action of the  subgroup $H_m \subset H_n$ on  the set of all Lagrangian subspaces $L \subset W_n$.
\begin{enumerate}
\item
The set of Lagrangians $L \subset W_n$ satisfying $L\cap V=0$ is a single $H_m$-orbit, and the stabilizer in $H_m$ of any such $L$ is trivial.
\item
If  $L \subset W_n$ is a Lagrangian  for which $L\cap V\neq 0$, its stabilizer  in $H_m$ is negligible.
\end{enumerate}
\end{lemma}

\begin{proof}
The proof of \cite[Lemma 2.1]{PSRa} establishes  a bijection between the set of Lagrangians $L \subset W_n$ and the set of triples $(Y,Y',\alpha)$ in which
\begin{itemize}
\item $Y \subset V$ and $Y' \subset V'$ are totally isotropic subspaces,
\item
$\alpha\colon Y^\perp / Y \iso  ( Y') ^\perp /Y'$ is an isometry.
\end{itemize}
The bijection sends a Lagrangian $L$ to the triple defined by $Y=L\cap V$ and $Y'= L\cap V'$ together with the isometry characterized  by
\[
\alpha( y ) = y' \iff y+y' \in L
\]
for all $y\in Y^\perp$ and $y'\in (Y')^\perp$.
Under this  bijection,  the action of the subgroup $H_m \times H_m' \subset H_n$ on the set of Lagrangians  translates to an action on triples:
 an element $(h,h')$ acts by
\[
(Y,Y',\alpha) \mapsto  (hY,h'Y' ,  h' \alpha h^{-1}),
\]
  in which  $h' \alpha h^{-1}$   is the composition
\[
(h Y)^\perp / hY \map{ h^{-1}}  Y^\perp / Y \map{\alpha}  ( Y') ^\perp /Y' \map{ h'}    ( h'Y') ^\perp / h'Y' ,
\]
and we are identifying $H_m \iso \mathrm{U} (V)$ and $H_m' \iso  \mathrm{U}(V')$.

Consider the set of all Lagrangians $L\subset W_n$ for which $L\cap V=0$.  These correspond to triples of the form
$(0,0,\alpha)$, where $\alpha\colon V \iso V'$ is an isometry.  The action of $H_m\iso  \mathrm{U}(V)$ on this set is simply transitive, proving (1).

To prove (2),  suppose $L$ corresponds to a triple  $(Y,Y',\alpha)$  with $Y\neq 0$.  The stabilizer of $L$ in $H_m$ is
\begin{align*}
& \{ h\in  H_m : (hY, Y' ,  \alpha h^{-1} )  = (Y,Y' ,\alpha)  \}  \\
& =
\{ h\in  H_m : hY=Y \mbox{ and } h=\mathrm{id} \mbox{ on } Y^\perp/Y \},
\end{align*}
which  is contained in the proper parabolic subgroup  stabilizing $Y$, and contains its  unipotent radical.
Hence this stabilizer is negligible.
\end{proof}

To complete the proof of Proposition \ref{prop:duplication unfolding}, we make use of  the bijection
\[
 P_n(F) \backslash H_n(F) \iso \{ \mbox{Lagrangian subspaces }L \subset W_n \}.
\]
If  $\gamma \in  P_n(F) \backslash H_n(F)$ corresponds to $L \subset W_n$, then $I_\gamma \subset H_m$ is equal to the stabilizer of $L$ in $H_m$.  By the lemma such an $I_\gamma$ is either trivial or negligible, and the $\gamma$ for which $I_\gamma=1$  form a single orbit under $H_m(F)$.

Having  noted above that $I_\delta=1$, we have now proved that $I_\gamma \subset H_m$ is negligible for any double coset
$
\gamma \in  P_n(F) \backslash H_n(F) / H_m(F)
$
with $\gamma \neq \delta$.
In other words, there exists a proper parabolic subgroup $Q_\gamma$ with unipotent radical $N_\gamma$ satisfying
$
N_\gamma \subset I_\gamma \subset Q_\gamma.
$
The integral
\[
\int_{ [ N_\gamma]  } \Phi_s( \gamma i ( g_1,g_2)   ) \,  f(g_1) \, d g_1
=
\Phi_s( \gamma i ( 1 ,g_2)   )
\int_{  [  N_\gamma ]   }  f(g_1) \, d g_1
\]
vanishes by the cuspidality of $f$, and therefore so does the integral indexed by $\gamma$ on the right hand side of \eqref{duplicate unfolding}.
\end{proof}

\begin{proof}[Proof of Theorem \ref{thm:duplication}]
Proposition  \ref{prop:duplication unfolding} implies the first equality in
\begin{align}\label{duplication punchline}
&  \chi(\det(g_2))^{-1}    \int_{ [H_m] } D( g_1 , g_2^\dagger   , s ,\chi )  f(g_1) \, d g_1    \\
  &=    \chi(\det(g_2))^{-1}  \int_{  H_m (\A_F)  } \Phi^\circ_s(   \delta i (   g_1 , g_2 )  ) f(g_1 ) \, d g_1   \nonumber  \\
   &=  \chi(\det(g_2))^{-1}  \int_{  H_m(\A_F)  } \Phi^\circ_s (  \delta  i  ( g_2 g_1 , g_2 )   ) f( g_2 g_1 ) \, d g_1   \nonumber  \\
  & =   \int_{  H_m (\A_F)  } \Phi^\circ_s(  \delta i (   g_1 , 1 )  ) f(g_2g_1 ) \, d g_1 . \nonumber
\end{align}
For the final equality we have used the observation that the composition
\[
H_m(\A_F) \map{g\mapsto \delta i  (g,g) \delta^{-1} } P_n (\A_F) \map{\eqref{parabolic character}} \A_{F'}^\times
\]
agrees with the determinant on $H_m(\A_F)$, a consequence of Lemma \ref{lem:true double}.

Viewing \eqref{duplication punchline} as an automorphic form in the variable $g_2\in H_m(\A_F)$, it is a $K_m$-fixed vector in the space of $\pi$.
Indeed,  the second integral in \eqref{duplication punchline} is right $K_m$-invariant because $\Phi_s^\circ$ is right $K_n$-invariant, and $i (K_m,K_m) \subset K_n$.
The final integral in \eqref{duplication punchline} is a linear combination of right translates of $f$, and hence lies in $\pi$.
 Up to scaling,  $f\in \pi$ is the unique $K_m$-fixed vector, and we deduce that
 \[
  \int_{  H_m (\A_F)  } \Phi^\circ_s(  \delta i (   g_1  , 1 )  )\,   f (g_2 g_1)   \, d g_1
  = c(s) \cdot f( g_2 )
\]
for some $c(s)$ independent of $g_2$.

To compute the scalar $c(s)$,  let $\breve{f} \in \breve{\pi}$ be the  $K_m$-fixed vector in the contragredient representation of $\pi$, normalized by    $\langle f, \breve{f}\rangle =1$.
Pairing both sides of the above equality against $\breve{f}$ results in
\[
 \int_{  H_m (\A_F)  } \Phi^\circ_s(  \delta  i (   h , 1 )  )  \,
 \langle \pi(h)  f , \breve{f}  \rangle  \, d h     = c(s) .
\]
The left hand side is the well-known doubling zeta integral studied in \cite{Li,PSRa}, see also \cite{HKS,Liu},  and we deduce
\[
c(s) =  \frac{  L( s + 1/2  ,\mathrm{BC}(\pi) \otimes \chi  )   }{ \mathscr{L}_n(s, \chi_0 )   } .
\]
This completes the proof of the duplication formula.
\end{proof}



\subsection{Degenerate coefficients in low rank}
\label{ss:degen}


We return to the assumption that the unramified Hecke character $\chi\colon\A_{F'}^\times \to \C^\times$ satisfies \eqref{hecke restriction}.

We are interested in degenerate coefficients of the Eisenstein series $E_T(g,s,\chi)$ on $H_2(\A_F)$ determined by the normalized spherical section $\Phi_s^\circ \in I_2(s,\chi)$.
More precisely, for a fixed
\[
T = \begin{pmatrix} 0 &  0 \\  0   & t \end{pmatrix} \in \Herm_2(F,\omega_F)
\]
with $t\in \Herm_1(F,\omega_F)$ nonzero, we will relate the   Fourier coefficient
\begin{align*}
E_T( g , s ,\chi )
& =
 \int_{   [ \Herm_2 ]   }     E ( n(b) g , s,\chi ) \,  \langle T,b\rangle_{\psi_0} \, db
\end{align*}
to the $t^\mathrm{th}$ coefficient of the Eisenstein series
\[
\underline{E}(h, s,\chi) = \sum_{\gamma \in P_1(F) \backslash H_1(F) }   \underline{\Phi}_s^\circ (\gamma h)
\]
on $H_1(\A_F)$ determined by the normalized spherical section $\underline{\Phi}_s^\circ \in I_1(s,\chi)$.

\begin{proposition}\label{prop:genus drop}
Fix $h\in H_1(\A_F)$ and
\[
m(\alpha) =  \left(  \begin{smallmatrix} \alpha \\ & \sigma(\alpha^{-1})  \end{smallmatrix}  \right) \in H_1(\A_F)
\]
with $\alpha\in \A_{F'}^\times$.  Recalling the standard embedding $i_0\colon H_1 \times H_1 \to H_2$ from \eqref{standard embedding},   we have the equality
\begin{align*}
&
 q^{ 2s \deg(\omega_X) }  \mathscr{L}_2(s,\chi_0) \cdot  E_T\big( i_0( m(\alpha) ,h)  , s,\chi \big)   \\
& =
  \chi(\alpha) |\alpha|_{F'}^{1+s}    q^{ 2s \deg(\omega_X) }   \mathscr{L}_2(s,\chi_0)  \cdot  \underline{E}_t \left( h ,   \frac{1}{2} + s  ,\chi \right)  \\
& \quad   +  \chi(\alpha) |\alpha|_{F'}^{ 1 -s }  q^{ -2s \deg(\omega_X) }  \mathscr{L}_2(-s,\chi_0) \cdot  \underline{E}_t \left( h ,  \frac{1}{2}  -s  ,\chi  \right) .
\end{align*}
Note that the right hand side is invariant under $s\mapsto -s$, in accordance with the functional equation \eqref{eis functional} satisfied by the left hand side.
\end{proposition}

\begin{proof}
By \cite[Lemma 5.2.5]{GaSa}, for all $g\in H_2(\A_F)$ we have
\begin{align}
E_T( g , s ,\chi)
& =
  \int_{ \Herm_1(\A_F) }
   \Phi_s^\circ \left(  i_0 \left( 1 ,   w_1^{-1}   n(b)  \right)  \cdot g  \right)\,
 \langle t,b\rangle_{\psi_0}  \, d b   \nonumber \\
   & \quad +
 \int_{  \Herm_2(\A_F)   }
  \Phi_s^\circ \left( w_2^{-1}  n(b)    g  \right)  \,   \langle T,b\rangle_{\psi_0}   \, d b .
   \label{genus drop split}
    \end{align}

The first integral in  \eqref{genus drop split} is easy to deal with.  After noting that
\begin{equation}\label{spherical genus drop}
\underline{\Phi}_{s+1/2}^\circ(h) = \Phi_s^\circ( i_0(1,h) )
\end{equation}
for all $h\in H_1(\A_F)$, we find
\begin{align}
&  \int_{ \Herm_1(\A_F) }
   \Phi^\circ_s \left(  i_0 \left( m(\alpha)  ,   w_1^{-1}   n(b)  h  \right)   \right)   \,
   \langle t,b\rangle_{\psi_0}   \, d b  \nonumber  \\
& =  \chi(\alpha) | \alpha |_{F'}^{s+1}   \int_{ \Herm_1(\A_F) }
   \Phi^\circ_s \left(  i_0 \left( 1  ,   w_1^{-1}   n(b) h  \right)   \right)   \,
     \langle t,b\rangle_{\psi_0}  \, d b  \nonumber  \\
 & =
    \chi(\alpha) | \alpha |_{F'}^{s+1}   \int_{ \Herm_1(\A_F) }   \underline{\Phi}^\circ_{s+1/2} ( w_1^{-1}  n(b)  h )    \,
     \langle t,b\rangle_{\psi_0}  \, d b   \nonumber   \\
& =  \chi(\alpha) | \alpha |_{F'}^{s+1}    \cdot  \underline{E}_t(h ,  s + 1/2   )     .  \label{easy genus drop}
\end{align}

The second integral in \eqref{genus drop split} is not quite so elementary, but is made explicit by the following lemma.

\begin{lemma}\label{lem:hard genus drop}
For all $h\in H_1(\A_F)$ we have
\begin{align*}
&   \int_{  \Herm_2(\A_F)   }
  \Phi^\circ_s \left( w_2^{-1}  n(b)    \cdot i_0( m(\alpha), h)   \right)  \,  \langle T,b\rangle_{\psi_0}  \, d b    \\
 & =
       \chi(\alpha) |\alpha|_{F'}^{-s+1}    q^{    - 4s   \deg(\omega_X)   }
  \frac{  \mathscr{L}_2(-s,\chi_0) }{ \mathscr{L}_2(s,\chi_0)    }
     \cdot \underline{E}_t(h , -s+1/2 ,\chi ),
\end{align*}
\end{lemma}

\begin{proof}
Define a subgroup  of $\Herm_2(\A_F)$ by
\[
U = \left\{ \begin{pmatrix}
 x & z \\ \sigma(z) & 0
\end{pmatrix}   : x \in \A_F \mbox{ and } z \in \A_{F'}\right\}.
\]
There is an obvious decomposition $\Herm_2(\A_F) \iso U \times \Herm_1(\A_F)$, where we view
$\Herm_1(\A_F) \subset \Herm_2(\A_F)$ as the subgroup of matrices with all entries outside the lower right corner equal to $0$.
The choices of Haar measures on $\Herm_1(\A_F)$ and $\Herm_2(\A_F)$ from Remark \ref{rem:measures} then determine a Haar measure on $U$ compatible with this product decomposition.

Consider the function of  $h\in H_1(\A_F)$ defined by
\[
f^\alpha_s(h) =
 \int_U
  \Phi^\circ_{s+1/2} \left( w_2^{-1}  n(u)    \cdot i_0( m(\alpha) , w_1^{-1} h)  \right) \, du .
  \]
As in \cite[Lemma 5.2.6(i)]{GaSa}, this function satisfies  $f^\alpha_s \in I_1(s,\chi)$, allowing us to  form the
 associated Eisenstein series $\underline{E}(h, f^\alpha_s)$ on $H_1(\A_F)$.
 It is clear from the definition that $f^\alpha_s$ is right invariant under $K_1$, and so we must have
 \[
 f^\alpha_s (h) = c(s) \cdot  \underline{\Phi}_s^\circ (h)
 \]
 for some function $c(s)$ independent of $h \in H_1(\A_F)$.

 We begin by computing
 \begin{align*}
&   \int_{  \Herm_2(\A_F)   }
  \Phi^\circ_s \left( w_2^{-1}    n(b)    \cdot i_0( m(\alpha), h)   \right)  \,  \langle T,b\rangle_{\psi_0}  \, d b     \\
  & =     \int_{ \Herm_1(\A_F) }   \int_U   \Phi^\circ_s \left( w_2^{-1}   n(u)   n\begin{pmatrix}   0 & 0  \\ 0 & b'   \end{pmatrix}
   \cdot i_0(  m(\alpha)  , h)  \right) \,   \langle T , b' \rangle_{\psi_0}  \, du       \,    db'         \\
 &=    \int_{ \Herm_1(\A_F)  }     \int_U  \Phi^\circ_s \left( w_2^{-1}   n(u)    \cdot i_0\left(   m(\alpha)    , n(b')  h    \right)   \right)
  \,   \langle T , b' \rangle_{\psi_0}  \, du   \,    db'       \\
 &=    \int_{ \Herm_1( \A_F )  }   f^\alpha_{s-1/2} ( w_1^{-1}  n(b')  h  )   \,   \langle T , b' \rangle_{\psi_0}  \,    db'      .
   \end{align*}
By \eqref{whit factor},  the final integral is the $t^\mathrm{th}$ coefficient of $\underline{E}(h, f^\alpha_{s-1/2})$, and hence
  \begin{align}\label{intertwine to eis}
&   \int_{  \Herm_2(\A_F)   }
  \Phi^\circ_s \left( w_2^{-1}    n(b)    \cdot i_0( m(\alpha), h)   \right)  \,  \langle T,b\rangle_{\psi_0}  \, d b      \\
& = c(s-1/2)  \cdot \underline{E}_t(h , s-1/2 ,\chi )  . \nonumber
   \end{align}

It remains  to compute $c(s)$.
 Recalling the intertwining operator \eqref{intertwining}, the same calculation as above (but with $T=0$) shows that
 \begin{align*}
 (  M_2(s)\Phi^\circ_s )( i_0( m(\alpha) ,h) )
&  = \int_{  \Herm_2(\A_F)   }
  \Phi^\circ_s \left( w_2^{-1}  n(b)    \cdot i_0( m(\alpha)  , h)   \right)  \,  d b    \\
 &=    \int_{ \Herm_1( \A_F )  }   f^\alpha_{s-1/2} ( w_1^{-1}  n(b')  h  )    \,    db'    \\
   & =  ( M_1(s-1/2)   f^\alpha_{s-1/2}  )   (h).
\end{align*}
On the other hand
 \begin{align*}
& (  M_2(s)\Phi^\circ_s )( i_0( m(\alpha) ,h) )   \\
& =  \chi(\alpha) |\alpha|_{F'}^{-s+1}  (  M_2(s)\Phi^\circ_s )( i_0( 1 ,h) ) \\
& \stackrel{\eqref{spherical intertwining}}{=}
  \chi(\alpha) |\alpha|_{F'}^{-s+1}  q^{     -  4s  \deg(\omega_X)   }
 \frac{  \mathscr{L}_2(-s,\chi_0 ) }{ \mathscr{L}_2(s,\chi_0 )  }    \Phi^\circ_{-s}( i_0(1,h) )   \\
 & \stackrel{\eqref{spherical genus drop}}{=}
    \chi(\alpha) |\alpha|_{F'}^{-s+1}   q^{     -  4s  \deg(\omega_X)   }
  \frac{  \mathscr{L}_2(-s,\chi_0 ) }{ \mathscr{L}_2(s,\chi_0 )  }   \underline{\Phi}_{-s+1/2}^\circ(h) .
\end{align*}
Combining these shows that
\[
M_1(s-1/2)   f^\alpha_{s-1/2}
 =
   \chi(\alpha) |\alpha|_{F'}^{-s+1}   q^{     -  4s  \deg(\omega_X)   }
  \frac{  \mathscr{L}_2(-s,\chi) }{ \mathscr{L}_2(s,\chi)  }   \underline{\Phi}_{-s+1/2}^\circ  ,
\]
and comparing with
\[
  M_1(s-1/2) \underline{\Phi}_{s-1/2}^\circ  \stackrel{\eqref{spherical intertwining}}{=}    q^{   (- 2s+1 )  \deg(\omega_X)   }
     \frac{  \mathscr{L}_1(-s + 1/2 ,\chi_0) }  { \mathscr{L}_1(s -1/2,\chi_0)   }   \cdot  \underline{\Phi}^\circ_{-s + 1/2}
\]
one deduces
\[
 c( s-1/2)
     =
       \chi(\alpha) |\alpha|_{F'}^{-s+1}    q^{    - ( 2s+1 )  \deg(\omega_X)   }
  \frac{  \mathscr{L}_2(-s,\chi_0) }{ \mathscr{L}_2(s,\chi_0)    }
     \frac  { \mathscr{L}_1(s -1/2,\chi_0)   }  {  \mathscr{L}_1(-s + 1/2 ,\chi_0) }  .
\]
To complete the proof of the lemma, substitute this formula into \eqref{intertwine to eis} and use the functional equation \eqref{eis functional} for $\underline{E}(h,s,\chi)$.
\end{proof}

To complete the proof of Proposition \ref{prop:genus drop},
take $g=i_0(m(a),h)$ in the equality  \eqref{genus drop split},  replace the first  integral  with the expression from \eqref{easy genus drop}, and replace the second integral with the expression from Lemma \ref{lem:hard genus drop}.
\end{proof}

We now translate Proposition \ref{prop:genus drop} into the language of geometric Fourier coefficients from \S \ref{ss:fourier}.
Fix two line bundles $\mathcal{E}_1, \mathcal{E}_2 \in \mathrm{Pic}(X')$ and an injective hermitian morphism $a_2\colon \mathcal{E}_2 \to \sigma^* \mathcal{E}_2^\vee$.  This data determines a pair $(\mathcal{E},a)$,  in which
\[
\mathcal{E} =  \mathcal{E}_1 \oplus \mathcal{E}_2 \map{  a = 0 \oplus a_2  }  \sigma^*\mathcal{E}_1^\vee \oplus  \sigma^*\mathcal{E}_2^\vee = \sigma^* \mathcal{E}^\vee.
\]
 Proposition \ref{prop:genus drop} is equivalent to the relation
\begin{align}\label{geometric genus drop}
&
 q^{ 2s \deg(\omega_X) }  \mathscr{L}_2(s,\chi_0) \cdot  E_{ ( \mathcal{E},a) } (s,\chi)  \\
& =
  \chi(\mathcal{E}_1)  q^{(1+s) \deg(\mathcal{E}_1)}   q^{ 2s \deg(\omega_X) }   \mathscr{L}_2(s,\chi_0)  \cdot  \underline{E}_{(\mathcal{E}_2 , a_2) } \left(  \frac{1}{2} + s  ,\chi \right)   \nonumber \\
& \quad   +   \chi(\mathcal{E}_1)  q^{(1-s) \deg(\mathcal{E}_1)}   q^{ -2s \deg(\omega_X) }   \mathscr{L}_2(-s,\chi_0)  \cdot  \underline{E}_{(\mathcal{E}_2 , a_2) } \left(  \frac{1}{2} - s  ,\chi \right)  \nonumber
\end{align}
for all pairs $(\mathcal{E},a)$ of the above form.


\section{Cycles on moduli spaces of shtukas}
\label{s:shtukas}


Fix an integer $n\ge 1$.
Now we turn to the study of  special cycle classes on the moduli space $\Sht^r_{\mathrm{U}(n)}$ of unitary shtukas, as in the introduction.

First we recall the arithmetic Siegel-Weil formula of Feng-Yun-Zhang, relating the degrees of special $0$-cycles on $\Sht^r_{\mathrm{U}(n)}$ to the Siegel Eisenstein series from \S \ref{s:automorphic}.
In  Conjecture \ref{conj:ASW} we propose a modest extension of this formula.  This  conjectural extension seems to be the minimum extra information that one needs   to extract a Gross-Zagier style intersection formula from the arithmetic Siegel-Weil formula.

In \S \ref{ss:low rank ASW}  we prove Conjecture \ref{conj:ASW} in a simple case, for special $0$ cycles on  $\Sht^r_{\mathrm{U}(2)}$ of a particular type.
Finally, in \S \ref{ss:main results},  we  put everything together to prove Theorems \ref{thm:simple duplication} and \ref{thm:GKZ} of  the introduction.


\subsection{The arithmetic Siegel-Weil formula}


As in \cite[\S 7]{FYZ1}, let
\[
 \cZ^r_\mathcal{E}(a) \to  \Sht_{\mathrm{U}(n)}^r
\]
be the naive special cycle determined by  a rank $n$ vector bundle $\mathcal{E}$ on $X'$, and a hermitian morphism  $a\colon\mathcal{E} \to \sigma^*\mathcal{E}^\vee$.
This naive special cycle has expected dimension $0$, and in   \cite[Definition 4.8]{FYZ2} one finds the the construction of a $0$-cycle class
\[
[ \cZ^r_\mathcal{E}(a) ] \in \mathrm{Ch}_0 ( \cZ^r_\mathcal{E}(a) ).
\]
In other words, we have now put ourselves in the $m=n$ case of   \eqref{FYZclass}.

The following \emph{arithmetic Siegel-Weil formula} is the main result of \cite{FYZ1}.
  It is the function field analogue of a formula conjectured by Kudla-Rapoport  \cite{KR},  and proved by Li-Zhang \cite{LZ}, relating the degrees of special cycles on (integral models of) unitary Shimura varieties to derivatives of Siegel Eisenstein series on quasi-split unitary groups.

\begin{theorem}[Feng-Yun-Zhang] \label{thm:ASW}
Let $(\mathcal{E},a)$ be as above.  If the hermitian morphism $a$ is injective, then the naive special cycle $\cZ^r_\mathcal{E}(a)$ is proper over $k$, and the special $0$-cycle class
$
[ \cZ^r_\mathcal{E}(a) ] \in \mathrm{Ch}_c^{r n} (  \Sht_{\mathrm{U}(n)}^r )
$
 satisfies
\begin{equation}\label{ASW}
\deg\, [ \cZ^r_\mathcal{E}(a) ]
  =  \frac{1}{ (\log q)^r }   \cdot
   \frac{q^{  \frac{n}{2}  d(\mathcal{E} )   }   }{   \chi( \det(\mathcal{E}  ) )      }   \cdot
    \frac{d^r}{ds^r} \Big|_{s=0}
   \widetilde{E}_{  ( \mathcal{E} ,a)  } ( s , \chi)
\end{equation}
for any unramified Hecke character
\[
\chi\colon\A_{F'}^\times \to \C^\times
\]
whose restriction to $\A_F^\times$ is $\eta^n$.
On the right hand side,
\begin{itemize}
\item  $d(\mathcal{E})$ is defined by \eqref{eq:dE},
\item  $\widetilde{E}(g,s,\chi)$  is  the renormalized Eisenstein series  on $H_n(\A_F)$  of \eqref{renormalized eisenstein},
\item
 $ \widetilde{E}_{  ( \mathcal{E} ,a)  } ( s , \chi)$ is its geometric Fourier coefficient in the sense of \S \ref{ss:fourier}.
 \end{itemize}
 \end{theorem}

\begin{proof}
The properness assertion is \cite[Proposition 7.13]{FYZ1}.  The equality \eqref{ASW} follows by combining the equalities
  \[
  \deg (    [ \cZ^r_\mathcal{E} ( a)  ]  )
  =  \frac{1}{ (\log q)^r  } \frac{d^r}{ds^r}\Big|_{s=0}
  q^{ s d(\mathcal{E} )}  \mathrm{Den}( q^{-2s}, \mathrm{coker}(a) )
  \]
  and
  \[
  \mathrm{Den}( q^{-2s}, \mathrm{coker}(a) )
  =
  \frac{q^{ (\frac{n}{2}-s) d(\mathcal{E})}  }{\chi(\det(\mathcal{E}))   }
  \widetilde{E}_{  ( \mathcal{E}  ,a)  } ( s , \chi)
  \]
  of \cite[\S 9.3 and \S 9.5]{FYZ2}.
\end{proof}

As explained in the introduction, properness of the naive special cycle $\cZ^r_\mathcal{E}(a)$ is needed to make sense of the degree \eqref{degree} appearing in the theorem.
Because lack of properness is the only obstruction to formulating the equality  \eqref{ASW} in general (i.e.~without the assumption that $a$ is injective), we make the following conjecture.

\begin{conjecture}\label{conj:ASW}
The equality \eqref{ASW} holds whenever the naive special cycle $\cZ^r_\mathcal{E}(a)$ is proper over $k$.
\end{conjecture}

Proving this in any case with $a$ not injective seems to require methods substantially different from those that go into the proof of Theorem \ref{thm:ASW}.


\subsection{A low rank case of Conjecture \ref{conj:ASW}}
\label{ss:low rank ASW}


We will prove  Conjecture \ref{conj:ASW} when $\mathcal{E}$ is a vector bundle on $X'$ of  rank $n=2$,  and $a\colon\mathcal{E} \to \mathcal{E}^\vee$ is a (not necessarily injective) hermitian morphism of a particular type.
 Some of the preparatory results needed for this result work in greater generality, so for the moment we work with any  positive integers $m\le n$.

\begin{lemma}\label{lem:lcitrunc}
  Let $\mathscr X$ be a connected quasi-smooth derived Artin stack over $k$ with classical truncation $i\colon X\hookrightarrow\mathscr X$. If $X$ is a local complete intersection  over $k$ of dimension  $\dim(X)=\operatorname{vd}(\mathscr{X})$, where $\operatorname{vd}(\mathscr X)$ denotes the virtual dimension of $\mathscr X$, then $\mathscr X$ is classical and $i\colon X\hookrightarrow\mathscr X$ is an isomorphism.
\end{lemma}

\begin{proof}
  The derived pullback of $i$ along any derived affine smooth chart $\Spec (R) \to\mathscr X$ is the canonical inclusion
  \[\Spec(\pi_0(R))\hookrightarrow \Spec (R),\]
  where $\pi_0(R)$ is a local complete intersection and $\dim(\pi_0(R))=\operatorname{vd}(R)$.
  As the hypotheses and conclusion of the lemma are both smooth-local in $\mathscr X$,  we may assume $\mathscr X=\Spec (R)$ and $X=\Spec(\pi_0(R))$.

  In this setting, the canonical map $R\twoheadrightarrow\pi_0(R)$ corresponding to $i$ induces an isomorphism on $\pi_0$, so the fiber of this map is 1-connective. By~\cite[Corollary 25.3.6.4]{Lur}, the relative cotangent complex $\bL_{\pi_0(R)/R}=\bL_{i}$ is 2-connective, i.e.\ of Tor-amplitude in degrees $\leq-2$. On the other hand, both $R$ and $\pi_0(R)$ are quasi-smooth, so $\bL_{R/k}$ and $\bL_{\pi_0(R)/k}$ are perfect of Tor-amplitude in degrees $\geq -1$, which implies that $\bL_{i}$ is perfect of Tor-amplitude in degrees $\geq -2$. Combining these two observations, it follows that $\bL_{i}\simeq M[2]$ for some locally free $\pi_0(R)$-module $M$. Finally,
  \[
  \rank (M)=\chi(\bL_{i})=\chi(\bL_{\pi_0(R)/k})-\chi(\bL_{R/k})=\dim(\pi_0(R))-\operatorname{vd}(R)=0,
  \]
   so $M$ is the zero module, and $\bL_i\simeq 0$. As $i$ also clearly induces an isomorphism on classical truncations, it follows by \cite[Corollary 25.3.6.6]{Lur} that $i$ is an isomorphism.
\end{proof}

\begin{lemma}\label{lem:LCIclasses}
Suppose $\mathcal{E}$ is rank $m$ vector bundle on $X'$, and  $a\colon\mathcal{E} \to \sigma^*\mathcal{E}^\vee$ is a hermitian morphism.  If  the naive special cycle $\cZ^r_{\mathcal{E}}(a)$ on  $\Sht^r_{\mathrm{U}(n)}$
is a local complete intersection  of dimension $r(n-m)$, then the  cycle class
\[
[ \cZ_{\mathcal{E}}^r(a) ] \in \Ch_{ r (n-m) } ( \cZ_{\mathcal{E}}^r(a) )
\]
 is the fundamental class of $\cZ_{\mathcal{E}}^r(a)$ in the usual sense.
\end{lemma}

\begin{proof}
The content  of \cite[Theorem 6.6, Remark 6.7]{FYZ2} is that the naive special cycle $\cZ^r_{\cE}(a)$ can be realized as the classical truncation of a quasi-smooth derived Artin stack $\sZ^r_{\cE}(a)$ of virtual dimension $r(n-m)$,   in such a way that the canonical isomorphism
\[
\Ch_{ r(n-m) }  ( \cZ^r_{\cE}(a) ) \iso \Ch_{ r(n-m) }  ( \sZ^r_{\cE}(a) )
\]
identifies the cycle class $[ \cZ_{\mathcal{E}}^r(a) ]$ of \cite[Definition 4.8]{FYZ2} with the
virtual fundamental class of $ \sZ^r_{\cE}(a)$.
The claim follows immediately from this, as $ \cZ^r_{\cE}(a) \iso  \sZ^r_{\cE}(a)$ by Lemma \ref{lem:lcitrunc}.
\end{proof}

\begin{lemma}\label{lem:shtu1proper}
  The canonical  morphism  $\Sht_{\mathrm{U}(1)}^r\to (X')^r$ recording the legs is finite \'etale.
  \end{lemma}

\begin{proof}
The cartesian diagram \cite[(2.2)]{FYZ2} expresses this morphism  as a pullback of the (finite \'etale)  Lang isogeny  $\Prym\to\Prym$, where $\Prym$ is the kernel of the norm map $\Pic_{X'} \to \Pic_X$.
\end{proof}

\begin{lemma}\label{lem:unm}
Suppose $\mathcal{E}$ is rank $m$ vector bundle on $X'$, and  $a\colon\mathcal{E} \iso \sigma^*\mathcal{E}^\vee$ is a hermitian isomorphism.   Let $\cZ^{r}_{\cE} (a) \to \Sht^r_{ \mathrm{U}(n) }$ be the associated naive special cycle.
  There is an isomorphism  of $k$-stacks
  \[
    \Sht_{\mathrm{U}(n-m)}^r \iso \cZ^{r}_{\cE} (a) .
  \]
\end{lemma}

\begin{proof}
For a $k$-scheme $S$, an object of  $\Sht_{ \mathrm{U}(n-m)}^r(S)$ is a chain of modifications
\begin{equation}\label{wee generic shtuka}
\mathcal{G}_0 \dashrightarrow\mathcal{G}_1\dashrightarrow \cdots \dashrightarrow \mathcal{G}_r \iso {}^\tau \mathcal{G}_0
\end{equation}
of vector bundles rank $n-m$ on $X'_S$,  endowed with compatible hermitian isomorphisms $h_i\colon\mathcal{G}_i \iso \sigma^* \mathcal{G}_i^\vee$, and satisfying some extra conditions \cite[Definition 6.6]{FYZ1}.
Abbreviating  $\mathcal{F}_i = \mathcal{E}_S \oplus \mathcal{G}_i$ for the orthogonal direct sum, the chain of modifications
\begin{equation}\label{generic shtuka}
\mathcal{F}_0 \dashrightarrow\mathcal{F}_1\dashrightarrow \cdots \dashrightarrow \mathcal{F}_r \iso {}^\tau \mathcal{F}_0
\end{equation}
defines an object of $\Sht_{ \mathrm{U}(n)}^r(S)$.  The canonical inclusions $\mathcal{E}_S \to \mathcal{F}_i$ define a lift of this point to $\cZ^{r}_{\cE} (a)$, and this  defines a morphism
\[
\Sht_{\mathrm{U}(n-m)}^r \to \cZ^{r}_{\cE} (a) .
\]

Conversely,  if one starts with a chain of modifications \eqref{generic shtuka} defining an object of $\Sht_{ \mathrm{U}(n)}^r(S)$, and a compatible family of isometric embeddings $\mathcal{E}_S \to \mathcal{F}_i$ defining a lift to $\mathcal{Z}_\mathcal{E}(a)(S)$, then  $\mathcal{E}_S \subset \mathcal{F}_i$  splits off an orthogonal complement $\mathcal{G}_i = \mathcal{E}_S^\perp$.  These define a chain of modifications \eqref{wee generic shtuka}, and hence an object of $\Sht_{ \mathrm{U}(n-m)}^r(S)$.  This provides an inverse to the above morphism.
\end{proof}

Recall from \cite[\S 4.3]{FYZ2} that  the moduli space $\Sht^r_{\mathrm{U}(n)}$ carries a collection of \emph{tautological line bundles} $\ell_1,\ldots, \ell_r \in \Pic( \Sht^r_{\mathrm{U}(n)})$, as well as a Chern class map
\[
c_1\colon\Pic( \Sht^r_{\mathrm{U}(n)}) \to \Ch^1( \Sht^r_{\mathrm{U}(n)}  ) .
\]
Denote by $p_1,\ldots, p_r\colon \Sht^r_{\mathrm{U}(n)} \to X'$ the morphisms recording the legs.

\begin{lemma}\label{lem:unn}
Suppose $\mathcal{E}_1$ is a line bundle on $X'$, $\mathcal{E}_2$ is a vector bundle of rank $n-1$, and
 $a_2\colon \mathcal{E}_2 \to \sigma^* \mathcal{E}_2^\vee$ is an injective hermitian morphism.
 If we endow the rank $n$ vector bundle $\mathcal{E}=\mathcal{E}_1\oplus \mathcal{E}_2$ with the hermitian morphism
\[
a = \begin{pmatrix} 0 & 0 \\ 0 & a_2 \end{pmatrix}\colon \mathcal{E} \to \sigma^*\mathcal{E}^\vee,
\]
then there is an isomorphism of $\Sht^r_{\mathrm{U}(n)}$-stacks
$
\cZ^r_{\mathcal{E}  }( a )  \iso  \cZ^r_{\mathcal{E}_2  }( a_2 )
$
such that the induced isomorphism of Chow groups identifies
\begin{equation}\label{unn class}
 [\cZ^r_{\mathcal{E}  }( a ) ] \in \Ch_0(\cZ^r_{\mathcal{E}  }( a ))
\end{equation}
with
\[
\left(\prod_{i=1}^r c_1(p_i^*\sigma^*\cE_1^{-1}\otimes\ell_i)\right)  \cdot
   [\cZ^r_{\mathcal{E}_2  }( a_2 ) ] \in \Ch_0(\cZ^r_{\mathcal{E}_2  }( a_2 )) .
\]
The intersection pairing
\[
\Ch_{r(n-1)} ( \Sht^r_{\mathrm{U}(n)} )  \times  \Ch_r(\cZ^r_{\mathcal{E}_2  }( a_2 ))
\to \Ch_0(\cZ^r_{\mathcal{E}_2  }( a_2 ))
\]
on the left is that of  \cite[\S 7.7]{FYZ1}.
\end{lemma}

\begin{proof}
As in the proof of Lemma \ref{lem:unm}, an object of $\cZ^r_{\mathcal{E}_2  }( a_2 )(S)$ is a chain of modifications
\begin{equation}\label{another generic shtuka}
\mathcal{F}_0 \dashrightarrow\mathcal{F}_1\dashrightarrow \cdots \dashrightarrow \mathcal{F}_r \iso {}^\tau \mathcal{F}_0
\end{equation}
of rank $n$ vector bundles on $X'_S$, endowed with compatible hermitian isomorphisms $h_i\colon \mathcal{F}_i \iso \sigma^* \mathcal{F}_i^\vee$, and a compatible family of    morphisms $\mathcal{E}_{2,S} \to \mathcal{F}_i$ respecting the hermitian structures.
If we extend  the domain of these morphisms to $\mathcal{E}_S$ by
\begin{equation}\label{proj map}
\mathcal{E}_S=
\mathcal{E}_{1,S} \oplus \mathcal{E}_{2,S} \map{\mathrm{proj}} \mathcal{E}_{2,S}  \to \mathcal{F}_i,
\end{equation}
we obtain an object of $\cZ^r_\mathcal{E}(a)(S)$.   This defines $\cZ^r_{\mathcal{E}_2  }( a_2 ) \to \cZ^r_{\mathcal{E} }( a )$.

For the inverse, if we start with an object of $\cZ^r_{\mathcal{E} }( a )(S)$ defined by a chain of modifications \eqref{another generic shtuka} and a compatible family of maps $\mathcal{E}_S \to \mathcal{F}_i$ respecting hermitian structures, the image of
$\mathcal{E}_{1,S} \to \mathcal{F}_i$ must be orthogonal both to itself and to the image of  the (injective, as $a_2$ is injective) morphism $\mathcal{E}_{2,S} \to \mathcal{F}_i$.
The subsheaf $\mathcal{E}_{2,S}^\perp \subset \mathcal{F}_i$ is a line bundle on which the hermitian form on $\mathcal{F}_i$ is nondegenerate, which therefore has no nonzero isotropic local sections.
Thus the morphism $\mathcal{E}_{1,S} \to \mathcal{F}_i$ is zero, and  $\mathcal{E}_S \to \mathcal{F}_i$ has the form \eqref{proj map} for some morphism $\mathcal{E}_{2,S} \to \mathcal{F}_i$.
This morphism defines an object of $\cZ^r_{\mathcal{E}_2  }( a_2 )(S)$.

The identification of cycle classes is now an  exercise in unpacking \cite[Definition 4.8]{FYZ2}.
Using the notation found there, there is a commutative diagram
\[
\xymatrix{
& {\mathcal{Z}^r_{\mathcal{E}_2} (a_2)  }  \ar[d]  \ar@{=}[rr] & &   {   \mathcal{Z}_\mathcal{E}^r (a) }  \ar[d] \\
{  \mathcal{Z}^r_{\mathcal{E}_2}  }    &    {  \mathcal{Z}_{ \mathcal{E}_2}^r[0]^\circ } \ar@{=}[r]  \ar[l] &  {  \mathcal{Z}_{\mathcal{E}/\mathcal{E}_1}^{ r,\circ }  }  \ar@{=}[r] &
{   \mathcal{Z}_\mathcal{E}^r[ \mathcal{E}_1]^\circ }  \ar[r]  & {  \mathcal{Z}_\mathcal{E}^r  }
}
\]
in which every arrow is an open and closed immersion, and the top horizontal equality is the isomorphism we have constructed above.
In \cite[Definition 4.4]{FYZ2} one finds the definition of a cycle class
\[
[ \mathcal{Z}_{\mathcal{E}/\mathcal{E}_1}^{r,\circ} ] \in \Ch_r(  \mathcal{Z}_{ \mathcal{E}/\mathcal{E}_1}^{r,\circ} )
\]
whose restriction to the open and closed substack $\mathcal{Z}^r_{\mathcal{E}_2}(a_2)$ is
 $[\cZ^r_{\mathcal{E}_2  }( a_2 ) ]$. On the other hand, the class \eqref{unn class} is defined as the restriction  of
\[
[ \mathcal{Z}_\mathcal{E}^r[ \mathcal{E}_1]^\circ ] \define \left(\prod_{i=1}^r c_1(p_i^*\sigma^*\cE_1^{-1}\otimes\ell_i)\right)  \cdot
  [ \mathcal{Z}_{\mathcal{E}/\mathcal{E}_1}^{r,\circ} ]\in \Ch_0(  \mathcal{Z}_{\mathcal{E}/\mathcal{E}_1}^{r,\circ}  )
\]
to the open and closed substack $\mathcal{Z}_\mathcal{E}^r(a)$.   The desired equality of cycle classes is immediate from this.
\end{proof}

We now specialize to the case  $n=2$, and combine Proposition \ref{prop:genus drop} with calculations of Feng-Yun-Zhang on $\Sht^r_{ \mathrm{U}(1) }$
to prove cases of Conjecture \ref{conj:ASW} on $\Sht^r_{ \mathrm{U}(2) }$.

\begin{theorem}\label{thm:intersection}
Suppose $\mathcal{E}=\mathcal{E}_1\oplus \mathcal{E}_2$ is a direct sum of line bundles on $X'$, endowed with a hermitian morphism of the form
\[
a = \begin{pmatrix} * & * \\ * & a_2 \end{pmatrix}\colon \mathcal{E} \to \sigma^*\mathcal{E}^\vee
\]
with $a_2\colon\mathcal{E}_2 \iso \sigma^*\mathcal{E}_2^\vee$ an isomorphism.  The naive special cycle
\[
 \cZ^r_\mathcal{E}(a) \to \Sht^r_{\mathrm{U}(2)}
\]
 is proper over $k$, and the degree of its associated $0$-cycle class
 \begin{equation}\label{intersection 0-cycle}
 [\cZ^r_{\mathcal{E}  }( a ) ] \in \Ch_0(\cZ^r_{\mathcal{E}  }( a ))
\end{equation}
 is given by the formula \eqref{ASW} predicted by Conjecture \ref{conj:ASW}.
\end{theorem}

\begin{proof}
If $a$ is injective, this is Theorem \ref{thm:ASW}.  If $a$ is not injective, the hypothesis that $a_2$ is an isomorphism implies that
$\mathrm{ker}(a) \subset \mathcal{E}$ is a complementary summand to $\mathcal{E}_2$.
As both sides of \eqref{ASW}   only depend on the isomorphism class of the pair $(\mathcal{E},a)$, we may  replace $\mathcal{E}_1$ with $\mathrm{ker}(a)$, and thereby reduce to the case
\[
a = \begin{pmatrix} 0 & 0 \\ 0 & a_2 \end{pmatrix} .
\]
For $a$ of this form, Lemmas \ref{lem:unm} and  \ref{lem:unn} provide us with isomorphisms
\begin{equation}\label{most simple cycle}
\Sht_{\mathrm{U}(1)}^r \iso \cZ_{\cE_2}^r(a_2) \iso \cZ_\cE^r(a)
\end{equation}
of $k$-stacks, all proper and smooth of dimension $r$ by Lemma \ref{lem:shtu1proper}.

The stack $\Sht^r_{ \mathrm{U}(1) }$ has its own naive special cycle
\[
\mathcal{Z}_{\mathcal{E}_1}(0) \iso \Sht^r_{ \mathrm{U}(1) },
\]
whose associated $0$-cycle class is defined \cite[Definition 4.8]{FYZ1} by
\begin{equation}\label{U(1)constant}
[ \cZ^r_{\mathcal{E}_1}(0) ] =    \prod_{i=1}^r c_1(p_i^*\sigma^*\cE_1^{-1}\otimes\ell_i) \in \Ch^r( \Sht^r_{\mathrm{U}(1)}).
\end{equation}

By Lemma \ref{lem:LCIclasses}, the cycle class
\[
  [\cZ^r_{\mathcal{E}_2  }( a_2 ) ] \in \Ch_r(\cZ^r_{\mathcal{E}_2  }( a_2 ))
\]
is just the usual fundamental class of the smooth $k$-stack $\cZ^r_{\mathcal{E}_2  }( a_2 )$, and it follows that
the first isomorphism in  \eqref{most simple cycle} identifies \eqref{U(1)constant} with
\[
\left(\prod_{i=1}^r c_1(p_i^*\sigma^*\cE_1^{-1}\otimes\ell_i)\right)  \cdot
   [\cZ^r_{\mathcal{E}_2  }( a_2 ) ] \in \Ch_0(\cZ^r_{\mathcal{E}_2  }( a_2 )) .
\]
In this last equality  $\ell_1,\ldots, \ell_r$ are the tautological line bundles on $\Sht^r_{ \mathrm{U}(n) }$, which pull-back under
\[
\Sht^r_{\mathrm{U}(1)} \iso \cZ^r_{\mathcal{E}_2  }( a_2 ) \to \Sht^r_{ \mathrm{U}(n) }
\]
 to the eponymous  line bundles on $\Sht^r_{\mathrm{U}(1)}$ appearing in  \eqref{U(1)constant}.

By Lemma \ref{lem:unn} and the preceding paragraph,  the composition   \eqref{most simple cycle} identifies \eqref{U(1)constant} with \eqref{intersection 0-cycle}.
In particular these $0$-cycles have the same degree.
The degree of \eqref{U(1)constant} is computed in \cite[Theorem 10.2]{FYZ2}, and we deduce the explicit formula
\begin{equation}\label{degen degree 1}
\deg\, [\cZ^r_{\mathcal{E}  }( a ) ]
=  \frac{2}{ (\log q)^{r}}    \cdot  \frac{d^r}{ds^r}\Big|_{s=0}   \left(   q^{d(\mathcal{E}_1) s}L(2s,\eta)  \right).
\end{equation}

Proposition~\ref{prop:genus drop}, rewritten as  \eqref{geometric genus drop},  expresses the $(\mathcal{E},a)$-coefficient of the spherical Eisenstein series $E(g,s,\chi)$  on $H_2(\A_F)$ in terms of the $(\mathcal{E}_2,a_2)$-coefficient of the spherical Eisenstein series $\underline{E}(g,s,\chi)$ on $H_1(\A_F)$.    This latter Fourier coefficient is known by the explicit formula
of Proposition \ref{prop:unr eisenstein coefficient}.  Combining these formulas, and simplifying using the global functional equation \eqref{L functional},  one finds
\[
\frac{   q^{  d(\mathcal{E}_1)     }   }{ \chi( \mathcal{E}_1  )  \chi( \mathcal{E}_2  )     } \cdot
 \widetilde{E}_{  ( \mathcal{E} ,a) } ( s , \chi)
=
q^{  d(\mathcal{E}_1)  s } L(2s,\eta) +       q^{ -   d(\mathcal{E}_1) s } L(-2s,\eta)  .
\]
Taking the $r^\mathrm{th}$ central derivative of both sides, and using the standing assumption that $r$ is even, we deduce
\begin{equation}\label{degen degree 2}
\frac{   q^{  d(\mathcal{E} )}   }   {   \chi(  \det( \mathcal{E} )  )  }
\cdot
\frac{d^r}{ds^r}\Big|_{s=0} \widetilde{E}_{  ( \mathcal{E} ,a)  } ( s , \chi)
=
2  \cdot
\frac{d^r}{ds^r}\Big|_{s=0}   \left(   q^{d(\mathcal{E}_1) s}L(2s,\eta) \right) .
\end{equation}
Note that we have used $d(\mathcal{E}_1) = d(\mathcal{E})$, which follows from
 $\mathcal{E}_2 \iso \sigma^* \mathcal{E}_2^\vee$.

Comparing \eqref{degen degree 1} with \eqref{degen degree 2} completes the proof.
\end{proof}


\subsection{Connection with the doubling kernel}
\label{ss:main results}


We now study  intersections of middle codimension cycles on $\Sht^r_{ \mathrm{U}(n)}$, and so assume  $n=2m$ is even.
Let \[\chi\colon\A_{F'}^\times \to \C^\times\] be an unramified Hecke character whose restriction to $\A_F^\times$ is trivial.

 Fix a pair $(\mathcal{E}_2,a_2)$ consisting of a rank $m$ vector bundle $\mathcal{E}_2$ on $X'$ and a hermitian morphism
$a_2\colon\mathcal{E}_2 \to \sigma^* \mathcal{E}_2^\vee$.  Assume that the associated naive special cycle
\begin{equation}\label{E_2 naive}
\cZ^r_{\mathcal{E}_2}(a_2) \to \Sht^r_{ \mathrm{U}(n) }
\end{equation}
is proper over $k$, so that it determines  a special cycle class
\[
[ \cZ^r_{\mathcal{E}_2}(a_2)  ]  \in   \Ch^{rm}_c(  \Sht^r_{ \mathrm{U}(n) }  )
\]
in the middle codimension Chow group with proper support.

Holding $(\mathcal{E}_2,a_2)$ fixed, consider pairs $(\mathcal{E}_1,a_1)$ with $\mathcal{E}_1$ another rank $m$ vector bundle on $X'$ and $a_1\colon\mathcal{E}_1 \to \sigma^* \mathcal{E}_1^\vee$ a hermitian morphism.
As in \cite[Lemma 7.7]{FYZ1}, there is a decomposition of the fiber product
\begin{equation}\label{naive intersection}
\cZ^r_{\mathcal{E}_1}(a_1)  \times_{  \Sht^r_{ \mathrm{U}(n) }  } \cZ^r_{\mathcal{E}_2}(a_2)
\iso
\bigsqcup_{  a = \left( \begin{smallmatrix} a_1 & * \\ * & a_2 \end{smallmatrix} \right) }  \cZ^r_{\mathcal{E}}(a )
\end{equation}
into open and closed substacks,
in which $\mathcal{E}=\mathcal{E}_1\oplus \mathcal{E}_2$,  and the disjoint union is over the same hermitian morphisms $a\colon\mathcal{E}\to \sigma^*\mathcal{E}^\vee$ as in   \eqref{eq:geometric doubling coefficients}.
As explained in \cite[\S 7.7]{FYZ1},  there is an intersection pairing
\[
\Ch_{rm} ( \cZ^r_{\mathcal{E}_1}(a_1) )  \times \Ch_{rm} ( \cZ^r_{\mathcal{E}_2}(a_2) )
\to \bigoplus_{  a = \left( \begin{smallmatrix} a_1 & * \\ * & a_2 \end{smallmatrix} \right) }  \Ch_0( \cZ^r_{\mathcal{E}}(a )  ) ,
\]
under which the special cycle classes of \eqref{FYZclass} satisfy the expected (but not obvious) intersection relation
\begin{equation}\label{linear invariance}
[ \cZ^r_{\mathcal{E}_1}(a_1) ] \cdot [  \cZ^r_{\mathcal{E}_2}(a_2) ]
=
\sum_{  a = \left( \begin{smallmatrix} a_1 & * \\ * & a_2 \end{smallmatrix} \right) } [  \cZ^r_{\mathcal{E}}(a ) ]
\end{equation}
 of \cite[Theorem 7.1]{FYZ2}.

The fiber product on the left hand side of \eqref{naive intersection} is finite over $\cZ^r_{\mathcal{E}_2}(a_2)$, and so is proper over $k$.
Hence each $\cZ^r_\mathcal{E}(a)$ appearing on the right hand side is also proper.
Using  the pushforwards   \eqref{bad push} and \eqref{good push}, we now   view \eqref{linear invariance} as an equality of cycle classes in the Chow group $\Ch_c^{rn}( \Sht^r_{\mathrm{U}(n)} )$ of $0$-cycles with proper support, with the intersection pairing on the left hand side  now understood as \eqref{proper pairing}.

Applying the degree map \eqref{degree} to both sides of \eqref{linear invariance}, Conjecture \ref{conj:ASW} implies the first equality in the  (conjectural) intersection formula
\begin{align} \label{double intersection}
&  \deg \big(  [ \cZ^r_{\mathcal{E}_1}(a_1) ] \cdot [  \cZ^r_{\mathcal{E}_2}(a_2) ]  \big)      \\
& \stackrel{?}{=}
  \frac{1}{ (\log q)^r }   \cdot
   \frac{q^{md(\mathcal{E}) }   }{ \chi( \det(\mathcal{E} ) )     }
 \sum_{  a = \left( \begin{smallmatrix} a_1 & * \\ * & a_2 \end{smallmatrix} \right) }   \frac{d^r}{ds^r} \Big|_{s=0}
   \widetilde{E}_{  ( \mathcal{E} ,a)  } ( s , \chi)    \nonumber   \\
&  \stackrel{  \eqref{eq:geometric doubling coefficients} }{=}    \frac{1}{ (\log q)^r}   \cdot
\frac{q^{md(\mathcal{E}_1) }    q^{md(\mathcal{E}_2) }
}{ \chi( \det(\mathcal{E}_1 ) )  \chi( \det(\mathcal{E}_2 ) )     }   \cdot
\frac{d^r}{ds^r} \Big|_{s=0}
\widetilde{D}_{(\mathcal{E}_1,a_1) ,  ( \mathcal{E}_2, a_2) }  ( s , \chi )   \nonumber
\end{align}
for any Hecke character
$
\chi\colon \A_{F'}^\times \to \C^\times
$
  whose restriction to $\A_F^\times$ is trivial.
On the right hand side is a (double) Fourier coefficient of the doubling kernel
\[
\widetilde{D}(g_1,g_2,s,\chi) \define  q^{  ns \deg ( \omega_X)  }    \cdot    \mathscr{L}_{n}(s,\chi_0)  \cdot  D ( g_1,g_2 ,s,\chi )
\]
from \eqref{Dkernel},  renormalized as in \eqref{renormalized eisenstein}.

\begin{theorem}\label{thm:final}
Suppose $\mathcal{E}_2$ is a line bundle on $X'$ (so $m=1$ and $n=2$), and  $a_2\colon \mathcal{E}_2 \iso \sigma^* \mathcal{E}_2^\vee$ is a hermitian isomorphism.
\begin{enumerate}
\item
The naive special cycle \eqref{E_2 naive} is proper over $k$.
\item
The equality  \eqref{double intersection} holds for every pair $(\mathcal{E}_1,a_1)$.
\item
There exists a $K_1$-fixed automorphic form $\mathscr{D} \in \mathcal{A}(H_1)$ whose geometric Fourier coefficients  are given by
\[
\mathscr{D}_{ (\mathcal{E}_1,a_1) } =
  \frac{ \chi(  \mathcal{E}_1)  } {  q^{   d(\mathcal{E}_1 ) } }  \,   \deg  \big(    [ \cZ^r_{\mathcal{E}_1}(a_1)]  \cdot [ \cZ^r_{\mathcal{E}_2}(a_2)] \big)
\]
for every pair $(\mathcal{E}_1,a_1)$ consisting of a line  bundle $\mathcal{E}_1$ on $X'$ and a hermitian morphism $a_1:\mathcal{E}_1 \to \sigma^* \mathcal{E}_1^\vee$.
\item
Assuming Conjecture \ref{conj:modularity}, the equality
\begin{align*}
& \deg \big(   \vartheta^{r,\chi} (f) \cdot [ \cZ^r_{\mathcal{E}_2}(a_2)] \big)  \\
& =
  f_{( \mathcal{E}_2 , -a_2)  }    \frac{    q^{d(\mathcal{E}_2) }    }{ (\log q)^r}
    \frac{d^r}{ds^r}\Big|_{s=0}\left( q^{  2s \deg ( \omega_X)  }    L( s +1/2  ,  \mathrm{BC}(\pi) \otimes \chi )\right),
\end{align*}
of Conjecture \ref{conj:GKZ} holds for every $K_1$-fixed $f\in \pi \subset \mathcal{A}(H_1)$.
\end{enumerate}
\end{theorem}

\begin{proof}
The properness claim  (1) follows from   Lemmas   \ref{lem:shtu1proper} and \ref{lem:unm}.

For (2),  we have explained above that  \eqref{double intersection} is a consequence of  Conjecture \ref{conj:ASW}, but one does not need to know that conjecture in full generality:
for a fixed  $(\mathcal{E}_2,a_2)$, to deduce \eqref{double intersection} one only needs to know Conjecture \ref{conj:ASW}  for those pairs $(\mathcal{E},a)$ of the form $\mathcal{E}=\mathcal{E}_1\oplus \mathcal{E}_2$ with
\[
 a = \left( \begin{matrix} a_1 & * \\ * & a_2 \end{matrix} \right) .
\]
For $(\mathcal{E}_2,a_2)$ satisfying our current hypotheses,  this is  Theorem \ref{thm:intersection}.

Analogues of (3) and (4) hold for any $n=2m$ and any pair $(\mathcal{E}_2,a_2)$ for which  \eqref{double intersection} holds, so that is the generality in which we prove them.

For any pair $(\mathcal{E}_1,a_1)$   the function $\widetilde{D}_{(\mathcal{E}_1,a_1) ,  ( \mathcal{E}_2, a_2) }  ( s , \chi )$ appearing in \eqref{double intersection} is the $(\mathcal{E}_1,a_1)$ Fourier coefficient of the
renormalized new way kernel
\[
\widetilde{D}_{ \square ,  ( \mathcal{E}_2, a_2) }  ( g_1, s , \chi )
=   q^{  ns \deg ( \omega_X)  }    \cdot    \mathscr{L}_{n}(s,\chi_0)  \cdot
D_{ \square ,  ( \mathcal{E}_2, a_2) }  ( g_1, s , \chi )
\]
of  \eqref{geometric new way kernel},  an automorphic form in the variable $g_1\in H_m(\A_F)$.
This proves that (3) holds with
\[
\mathscr{D}(g_1) \define
 \frac{1}{ (\log q)^r}   \cdot
\frac{   q^{md(\mathcal{E}_2) }  }{   \chi( \det(\mathcal{E}_2 ) )     }   \cdot
\frac{d^r}{ds^r} \Big|_{s=0}
\widetilde{D}_{ \square ,  ( \mathcal{E}_2, a_2) }  ( g_1, s , \chi ) ,
\]
as the equality \eqref{double intersection} says precisely that
\[
 \mathscr{D}_{ (\mathcal{E}_1,a_1)}
=
 \frac{ \chi( \det(\mathcal{E}_1 ) ) }{q^{md(\mathcal{E}_1) }} \deg \big(  [ \cZ^r_{\mathcal{E}_1}(a_1) ] \cdot [  \cZ^r_{\mathcal{E}_2}(a_2) ]  \big)    .
 \]

This last equality also shows that  the automorphic form $Z^{r,\chi}$ of Conjecture \ref{conj:modularity} satisfies
\[
\mathscr{D}(g_1) =
  \deg \big(  Z^{r,\chi} (g_1)   \cdot [  \mathcal{Z}_{\mathcal{E}_2}^r(a_2) ]  \big),
\]
because both sides have the same geometric Fourier coefficients.
Recalling the definition \eqref{arith theta} of the arithmetic theta lift, for any $K_m$-fixed cusp form $f\in \pi \subset \mathcal{A}(H_m)$ we now compute
\begin{align*}
& \deg \big(   \vartheta^{r,\chi} (f) \cdot [ \mathcal{Z}_{\mathcal{E}_2}^r(a_2)] \big)  \\
&=
 \int_{[H_m]}  f(g_1)   \, \deg \big(   Z^{r,\chi} (g_1) \cdot [ \mathcal{Z}_{\mathcal{E}_2}^r(a_2)]  \big)  \, dg_1  \\
 & =
  \int_{[H_m]}  f(g_1)   \mathscr{D}(g_1)  \, dg_1  \\
 &  =
  \frac{1}{ (\log q)^r}   \cdot
\frac{    q^{md(\mathcal{E}_2) }
}{  \chi( \det(\mathcal{E}_2 ) )     }   \cdot
\frac{d^r}{ds^r} \Big|_{s=0}
 \int_{[H_m]}  f(g_1)
\widetilde{D}_{ \square ,  ( \mathcal{E}_2, a_2) }  ( g_1 , s , \chi )  \, dg_1 .
\end{align*}
Theorem \ref{thm:duplication},  restated as  \eqref{geometric new way}, shows that
 \begin{align*}
  &\int_{ [H_m ] }   f(g_1 ) \widetilde{D}_{\square, ( \mathcal{E}_2,a_2) }( g_1   ,s,\chi )   \, d g_1  \\
    & =
 \chi(\det(\mathcal{E}_2))    \cdot   f_{( \mathcal{E}_2 , -a_2)  }     \cdot   q^{  ns \deg ( \omega_X)  }
  L( s +1/2  , \mathrm{BC}(\pi)  \otimes \chi )      ,
\end{align*}
and we deduce
\begin{align*}
& \deg \big(   \vartheta^{r,\chi} (f) \cdot [ \cZ^r_{\mathcal{E}_2}(a_2)] \big)  \\
& =
  f_{( \mathcal{E}_2 , -a_2)  }    \frac{    q^{md(\mathcal{E}_2) }    }{ (\log q)^r}
    \frac{d^r}{ds^r}\Big|_{s=0}\left( q^{  ns \deg ( \omega_X)  }    L( s +1/2  ,  \mathrm{BC}(\pi) \otimes \chi )\right).
\end{align*}
This both proves (4), and shows more generally that Conjecture \ref{conj:GKZ} is a consequence of  the conjectural equality \eqref{double intersection}.
\end{proof}

\bibliographystyle{alpha.bst}


\end{document}